\documentclass[12pt]{amsart}
\usepackage{amsmath}
\usepackage{amsfonts}
\usepackage{amssymb}
\usepackage{amsthm}
\usepackage{bbm}
\usepackage{bm}
\usepackage{xcolor}
\usepackage{pgffor}

\theoremstyle{plain}
\newtheorem{theorem}{Theorem}
\newtheorem{lemma}[theorem]{Lemma}
\newtheorem{corollary}[theorem]{Corollary}
\newtheorem{proposition}[theorem]{Proposition}
\newtheorem*{proposition*}{Proposition}

\theoremstyle{definition}

\theoremstyle{remark}
\newtheorem{remark}[theorem]{Remark}
\theoremstyle{definition}
\newtheorem*{remark*}{Remark}

\newcommand{\red}[1]{#1} %{{\textcolor{red}{#1}}}

\newcommand{\Z}{\mathbb{Z}}
\newcommand{\R}{\mathbb{R}}
\newcommand{\C}{\mathbb{C}}
\newcommand{\N}{\mathbb{N}}

\newcommand{\e}{{\rm e}}
\newcommand{\1}{{\bf 1}}

\newcommand{\cD}{{\mathcal D}}

\newcommand{\bc}{{\bm c}}
\newcommand{\bd}{{\bm d}}
\newcommand{\bn}{{\bm n}}
\newcommand{\br}{{\bm r}}
\newcommand{\bs}{{\bm s}}

\newcommand{\bC}{{\bm C}}
\newcommand{\bD}{{\bm D}}
\newcommand{\bN}{{\bm N}}
\newcommand{\bR}{{\bm R}}
\newcommand{\bS}{{\bm S}}

\newcommand{\eps}{\varepsilon}
\newcommand{\vphi}{\varphi}

\newcommand{\df}{\mathop{}\!\mathrm{d}}
\newcommand{\qt}{{\widetilde q}}
\newcommand{\ud}{{\mathfrak u}}

\renewcommand{\bar}{\overline}
\renewcommand{\hat}{\widehat}
\newcommand{\abs}[1]{\left| #1 \right|}

\renewcommand{\mod}[1]{\ ({\rm mod\ }#1)}

\newcommand{\ssum}[1]{\sum_{\substack{#1}}}
\newcommand{\summ}[2]{\mathop{\foreach \n in {1, ..., #1}{\sum}}_{\substack{#2}}}
\newcommand{\pmat}[1]{\begin{pmatrix} #1 \end{pmatrix}}

\DeclareMathOperator{\supp}{supp}
\DeclareMathOperator{\cond}{cond}
\renewcommand\Re{\operatorname{Re}}

\title{One-level density estimates for Dirichlet $L$-functions with extended support}

\author{Sary Drappeau}
\address{Aix Marseille Universite, CNRS, Centrale Marseille, I2M UMR 7373, 13453 Marseille, France}
\email[Sary Drappeau]{sary-aurelien.drappeau@univ-amu.fr}

\author{Kyle Pratt}
\address{All Souls College, University of Oxford, UK}
\email[Kyle Pratt]{Kyle.Pratt@maths.ox.ac.uk}

\author{Maksym Radziwi{\l\l}}
\address{Department of Mathematics, Caltech, 1200 E California BLVD, Pasadena, CA 91125, USA}
\email[Maksym Radziwi{\l\l}]{maksym.radziwill@gmail.com}

\date{\today}


\begin{document}

\begin{abstract}
  We estimate the $1$-level density of low-lying zeros of $L(s, \chi)$ with $\chi$ ranging over primitive Dirichlet characters of conductor $\in [Q/ 2, Q]$ and for test functions whose Fourier transform is supported in $(-2 - \frac{50}{1093}, 2 + \frac{50}{1093})$. Previously any extension of the support past the range $(-2, 2)$ was only known conditionally on deep conjectures about the distribution of primes in arithmetic progressions, beyond the reach of the Generalized Riemann Hypothesis (e.g Montgomery's conjecture). Our work provides the first example of a family of $L$-functions in which the support is unconditionally extended past the ``diagonal range'' that follows from a straightforward application of the underlying trace formula (in this case orthogonality of characters). We also highlight consequences for non-vanishing of $L(s, \chi)$. %

  \end{abstract}

\maketitle
\section{Introduction}

Motivated by the problem of establishing the non-existence of Siegel zeros (see \cite{ConreyIwaniec} for details), Montgomery \cite{Montgomery} investigated in 1972 the \textit{vertical} distribution of the zeros of the Riemann zeta-function. He showed that under the assumption of the Riemann Hypothesis, for any smooth function $f$ with $\text{supp } \hat{f} \subset (-1,1)$,
\begin{equation} \label{eq:montgomery}
\lim_{T \rightarrow \infty} \frac{1}{N(T)} \sum_{T \leq \gamma, \gamma' \leq 2T} f \Big ( \frac{\log T}{2\pi} \cdot (\gamma - \gamma') \Big ) = \int_{\mathbb{R}} f(u) \cdot \Big ( \delta(u) + 1 - \Big ( \frac{\sin 2\pi u}{2\pi u} \Big )^2 \Big ) du
\end{equation}
where $N(T)$ denotes the number of zeros of the Riemann zeta-function up to height $T$ and $\gamma, \gamma'$ are ordinates of the zeros of the Riemann zeta-function, and~$\delta(u)$ is a Dirac mass at~$0$.  Dyson famously observed that the right-hand side coincides with the pair correlation function of eigenvalues of a random Hermitian matrix.

Dyson's observation leads one to conjecture that the spacings between the zeros of the Riemann zeta-function are distributed in the same way as spacings between eigenvalues of a large random Hermitian matrix. Subsequent work of Rudnick-Sarnak \cite{RudnickSarnak} provided strong evidence towards this conjecture by computing (under increasingly restrictive conditions) the $n$-correlations of the zeros of any given automorphic $L$-function. Importantly the work of Rudnick-Sarnak suggested that the distribution of the zeros of an automorphic $L$-function is universal and independent of the distribution of its coefficients \cite{RudnickSarnak1996}. 

For number theoretic applications, the distribution of the so-called ``low-lying zeros'', that is zeros close to the central point is particularly interesting (see e.g \cite{HeathBrown, Young} for various applications; see also \cite{GS18} and \cite{Wat21}, for instance, for results in a different direction). Following the work of Katz-Sarnak \cite{KatzSarnak} and Iwaniec-Luo-Sarnak \cite{IwaniecLuoSarnak}, we believe that the distribution of these low-lying zeros is also universal and predicted by only a few random matrix ensembles (which are either symplectic, orthogonal or unitary). 

Specifically the work of Katz-Sarnak suggests that for any smooth function $\phi$ and any natural ``family'' of automorphic objects $\mathcal{F}$, 
\begin{equation} \label{eq:conjectured} 
\frac{1}{\# \mathcal{F}} \sum_{\pi \in \mathcal{F}} \sum_{\gamma_\pi} \phi \Big ( \frac{\log \mathfrak{c}_{\pi}}{2\pi} \cdot \gamma_\pi \Big ) \underset{\# \mathcal{F} \rightarrow \infty}{\longrightarrow} \int_{\mathbb{R}} \phi(x) K_{\mathcal{F}}(x) dx,
\end{equation}
where $\mathfrak{\gamma}_{\pi}$ are ordinates of the zeros of the L-function attached to $\pi$, $\mathfrak{c}_{\pi}$ is the analytic conductor of $\pi$ and $K_{\mathcal{F}}(x)$ is a function depending only on the ``symmetry type'' of $\mathcal{F}$. One may wish to consult \cite{IwaniecLuoSarnak} and \cite{SarnakShinTemplier2016} for a more detailed discussion.

There is a vast literature providing evidence for \eqref{eq:conjectured} (see \cite{survey}). Similarly to Montgomery's result \eqref{eq:montgomery} all of the results in the literature place a restriction on the support of the Fourier transform of $\phi$. This restriction arises from the limitations of the relevant trace formula (in some families it is not always readily apparent what this relevant trace formula is). 
In practice an application of the trace formula gives rise to so-called ``diagonal'' and ``off-diagonal'' terms. Trivially bounding the off-diagonal terms corresponds to what we call a ``straightforward'' application of the trace formula.

A central yet extremely difficult problem is to extend the support of $\hat{\phi}$ beyond what a ``straightforward'' application of the trace formula gives. In fact most works in which the support of $\hat{\phi}$ has been extended further rely on the assumption of various deep hypotheses about primes that sometimes lie beyond the reach of the Generalized Riemann Hypothesis (GRH). 

For example, Iwaniec-Luo-Sarnak show that in the case of holomorhic forms of even weight $\leq K$ one obtains unconditionally a result for $\hat{\phi}$ supported in $(-1,1)$ and that under the assumption of the Generalized Riemann Hypothesis this can be enlarged to $(-2, 2)$ (it is observed in \cite{DevinFiorilliSodergren} that assuming GRH only for Dirichlet $L$-functions is sufficient). Iwaniec-Luo-Sarnak also show that this range can be pushed further to $\text{supp } \hat{\phi} \subset (-22/9, 22/9)$ under the additional assumption that, for any $c \geq 1$, $(a,c) = 1$ and $\varepsilon > 0$, 
$$
\sum_{\substack{p \leq x \\ p \equiv a \pmod{c}}} e(2 \sqrt{p} / c) \ll_{\varepsilon} x^{1/2 + \varepsilon}.
$$
A similar behaviour is observed on low-lying zeros of dihedral~$L$-functions associated to an imaginary quadratic field~\cite{Fouvry2003}, where an extension of the support is shown to be equivalent to an asymptotic formula on primes with a certain splitting behaviour.

Assuming GRH, Brumer \cite{Brum92} studied the one-level density of the family of elliptic curves and proved a result for test functions supported in $(-5/9,5/9)$; this corresponds to the ``diagonal'' range for this family. Heath-Brown \cite{HeathBrown} improved this range to $(-2/3,2/3)$, and Young \cite{Young} pushed the support to $(-7/9,7/9)$. One-level density estimates for this family have deep implications for average ranks of elliptic curves. In particular, the work of Young was the first to show that, under some reasonable conjectures, a positive proportion of elliptic curves have rank~$0$ or~$1$ and thus satisfy the rank part of the Birch and Swinnerton-Dyer conjecture\footnote{A stronger conclusion was later reached unconditionally by Bhargava and Shankar~\cite{Bhargava2015} through other methods.}.

As another example, it follows for instance from minor modifications of \cite{HughesRudnick,ChandeeRadziwill} that in the family of primitive Dirichlet characters of modulus $\leq Q$ one can estimate 1-level densities unconditionally for $\phi$ with $\hat{\phi}$ supported in $(-2,2)$.\footnote{This is in fact the $GL(1)$ analogue of the result of Iwaniec-Luo-Sarnak for holomorphic forms.} As a by-product of work of Fiorilli-Miller~\cite[Theorem~2.8]{FiorilliMiller}, it follows that for any~$\delta\in(0, 2)$, this support can be enlarged to $(-2 -\delta, 2 + \delta)$ under the following ``de-averaging hypothesis''
\begin{equation} \label{eq:deavg}
  \sum_{Q /2 \leq q \leq Q} \Big | \sum_{\substack{p \leq x \\ p \equiv 1 \pmod{q}}} \log p - \frac{x}{\varphi(q)} \Big |^2 \ll Q^{-\delta/2} \sum_{Q /2 \leq q \leq Q} \sum_{(a,q) = 1} \Big | \sum_{\substack{p \leq x \\ p \equiv a \pmod{q}}} \log p - \frac{x}{\varphi(q)} \Big |^2.
\end{equation}

In this paper we give a first example of a family of $L$-functions in which we can \textit{unconditionally} enlarge the support past the ``diagonal'' range that follows from a straightforward application of the trace formula (in this case orthogonality of characters).

\begin{theorem} \label{th:main}
  Let $\Phi$ be a smooth function compactly supported in $[1/2, 3]$, and $\phi$ be a smooth function such that $\text{supp } \widehat{\phi} \subset ( -2 - \frac{50}{1093}, 2 + \frac{50}{1093})$. Then, as $Q \rightarrow \infty$, 
  \begin{equation}
    \label{eq:main}
    \sum_{q} \Phi \Big ( \frac{q}{Q} \Big ) \sum_{\substack{\chi\pmod{q} \\ \textup{primitive}}} \sum_{\gamma_\chi} \phi \Big ( \frac{\log Q}{2\pi} \gamma_{\chi} \Big ) = \widehat{\phi}(0) \sum_{q} \Phi \Big ( \frac{q}{Q} \Big ) \sum_{\substack{\chi \pmod{q} \\ \textup{primitive}}} 1 + o(Q^2).
  \end{equation}
  Here $\tfrac 12 + i \gamma_{\chi}$ correspond to non-trivial zeros of $L(s, \chi)$ and since we do not assume the Generalized Riemann Hypothesis we allow the $\gamma_{\chi}$ to be complex. 
\end{theorem}

\begin{remark*}
In stating the theorem we have, for technical simplicity, made a suitable approximation to the conductor $\mathfrak{c}_\pi$ appearing in \eqref{eq:conjectured}.
\end{remark*}

Note that~$\phi$, initially defined on~$\R$, is analytically continued to~$\C$ by compactness of~$\supp\hat\phi$. Our arguments can be adapted to show that if~$\text{supp } \widehat{\phi} \subset ( -2 - \frac{50}{1093}+\eps, 2 + \frac{50}{1093}-\eps)$ for some~$\eps>0$, then the error term in~\eqref{eq:main} is~$O(Q^{2-\delta})$ with~$\delta = \delta(\eps)$, up to altering slightly the main terms: after applying the explicit formula as in section \ref{section:explicit formula}, include the terms of order $\asymp Q^2/\log Q$ into the main term instead of treating them as error terms.

We remark that we make no progress on the ``de-averaging hypothesis''~\eqref{eq:deavg} of Fiorilli-Miller, which remains a difficult open problem. We estimate the original sum over primes in arithmetic progressions, on average over moduli, by a variant of an argument of Fouvry~\cite{Fouvry1985} and Bombieri-Friedlander-Iwaniec~\cite{BombieriFriedlanderEtAl1986} which is based on Linnik's dispersion method. The GRH will be dispensed with by working throughout, as in~\cite{Drappeau2015}, with characters of large conductors.

The asymptotic formula~\eqref{eq:main} is expected to hold true without the extra averaging over~$q$. This extra averaging over~$q$, and the cancellation of arguments which comes along, play an important role in our arguments.

If the GRH is true for Dirichlet~$L$-functions, then let any $0 < \kappa < \frac{50}{1093}$ be fixed, and let~$\lambda>1$ be small enough that~$\kappa' := 2(\lambda-1)+\lambda\kappa \in (0, \frac{50}{1093})$ as well. Defining
$$
{\tilde \phi}(x) = \lambda \Big ( \frac{\sin \pi (2 + \kappa)x}{\pi (2 + \kappa)x} \Big )^2, \qquad \phi = {\tilde \phi} \ast u
$$
where~$u$ is a smooth, positive approximation of unity such that~$\phi(0) \geq \lambda^{-1} \tilde\phi(0) = 1$, and using the inequality
$$
1 - \sum_{\gamma_{\chi}} \phi \Big ( \frac{\log Q}{2\pi} \gamma_{\chi} \Big ) \leq \mathbf{1} \Big ( L(\tfrac 12, \chi) \neq 0 \Big ) ,
$$
we deduce from Theorem \ref{th:main} that the proportion of non-vanishing $L(\tfrac 12, \chi)$ with $\chi$ ranging over primitive characters of conductor in $[Q / 2, Q]$ is at least $1 - \lambda (2 + \kappa')^{-1} = 1 - (2+\kappa)^{-1}$ for any $\kappa < \tfrac{50}{1093}$. We record this consequence in the Corollary below. 
\begin{corollary}\label{cor:nonvan} Let $\varepsilon \in (0, 10^{-7})$. 
  Assume the Generalized Riemann Hypothesis for Dirichlet $L$-functions. Then for all $Q$ large enough, the proportion of primitive characters $\chi$ with modulus $\in [Q/ 2, Q]$ for which
  $$
  L(\tfrac 12, \chi) \neq 0
  $$
  is at least
  $$
  \frac{1}{2} + \frac{25}{2236} - \varepsilon > 0.51118.
  $$
\end{corollary}

Corollary~\ref{cor:nonvan} is related to a recent result of Pratt \cite{Pratt2018} who showed unconditionally that the proportion of non-vanishing in this family is at least $0.50073$. We note that both the arguments of~\cite{Pratt2018} and those presented here eventually rely on bounds of Deshouillers-Iwaniec~\cite{DeshouillersIwaniec1982a} on cancellation in sums of Kloosterman sums.

\subsection*{Notations}

We call a map~$f:\R_+\to\C$ a \emph{test function} if~$f$ is smooth and supported inside~$[\frac12, 3]$.

For~$w\in\N$, $n\in\Z$ and~$R\geq 1$, we let
$$ \ud_R(n, w) := \1_{n\equiv 1 \mod{w}} -  \frac1{\vphi(w)}\ssum{\chi\mod{w} \\ \cond(\chi)\leq R} \chi(n). $$
Note the trivial bound
\begin{equation}
\abs{\ud_R(n, w)} \ll \1_{n\equiv 1\mod{w}} + \frac{R\tau(w)}{\vphi(w)}.\label{eq:ud-triv}
\end{equation}

The symbol~$n\sim N$ in a summation means~$n\in[N, 2N)\cap\Z$. We say that a sequence~$(\alpha_n)_n$ is \emph{supported at scale~$N$} if~$\alpha_n=0$ unless~$n\sim N$.

The letter~$\eps$ will denote an arbitrarilly small number, whose value may differ at each occurrence. The implied constants will be allowed to depend on~$\eps$.

\subsection*{Acknowledgments}
Part of this work was conducted while the second author was supported by the National Science Foundation Graduate Research Program under grant number DGE-1144245. The third author acknowledges the support of a Sloan fellowship and NSF grant DMS-1902063.
The authors thank the anonymous referee for helpful remarks, and Jared Lichtman for helpful discussions on Proposition 6.

\section{Proof of Theorem~\ref{th:main}}

\subsection{Lemmas on primes in arithmetic progressions}

We will require two results about primes in arithmetic progressions. The first is a standard estimate, obtained from an application of the large sieve. 
\begin{lemma}
  Let~$A>0$, $X, Q, R \geq 2$ satisfy~$1 \leq R \leq Q$ and~$X \geq Q^2/(\log Q)^A$, and~$f$ be a test function with~$\|f^{(j)}\|_\infty \ll_j 1$. Then
  \begin{equation}
    \label{eq:bound-delta-LS}
    \sum_{q\leq Q} \abs{\sum_{n\in \N} f\Big(\frac nX\Big) \Lambda(n) \ud_R(n, q)} \ll Q (\log Q)^{O(1)} \sqrt{X}\Big(1 + \frac{\sqrt{X}}{RQ} + \frac{X^{3/8}}Q\Big).    
  \end{equation}
  The implied constant depends at most on~$A$ and the implied constants in the hypothesis.
\end{lemma}
\begin{proof}
  By Heath-Brown's combinatorial formula for primes~\cite[Proposition~13.3]{IwaniecKowalski2004} (with~$K=2$), we restrict to proving the bound with~$\Lambda(n)$ replaced by convolutions of type~I and~II, of the shape
  \begin{align*}
    \summ{2}{n = m \ell \\ m\sim M} \alpha_m & \qquad (M \ll X^{1/4}), \\
    \summ{2}{n = m \ell \\ m\sim M} \alpha_m \beta_{\ell} & \qquad (X^{1/4} \ll M \ll X^{3/4}),
  \end{align*}
  where~$\abs{\alpha_m} \ll (\log X)\tau_4(m)$ and the analogous bound holds for~$\beta_\ell$; here we noted that if~$m_1\leq m_2\leq \sqrt{X}$ and~$m_1m_2 > X^{1/4}$, then either~$X^{1/4} < m_1m_2 \leq X^{3/4}$ or~$X^{1/4} \leq m_1 \ll X^{1/2}$. We treat the type~I case by the Poly\'{a}-Vinogradov inequality~\cite[Theorem~12.5]{IwaniecKowalski2004}, getting a bound~$O(M R^{3/2}(\log Q)^{O(1)})$. We treat the type~II case by the large sieve~\cite[Theorem~17.4]{IwaniecKowalski2004}, getting a contribution~$O(\sqrt{X}(\log Q)^{O(1)}(Q + \sqrt{M} + \sqrt{X/M} + \sqrt{X}R^{-1}))$.
\end{proof}

The second estimate is substantially deeper and we defer its proof to Section \ref{se:prop}.

\begin{proposition}\label{prop:primes-ap}
  Let~$\kappa \in (0, \frac{50}{1093})$ and~$\eps>0$. Let~$\Psi$ and~$f$ be test functions, $A>0$, $X, Q, W, R\geq 1$, and~$b\in\N$. Assume that
  \begin{equation*}
    \begin{aligned}
      \frac{Q^2}{(\log Q)^A} \ll X \ll Q^{2+\kappa}, \quad & \quad X^{11/20} Q^{-1} \leq R\leq Q^{2/3} X^{-2/9}, \\
      b \leq Q^\eps \quad & \quad Q^{1-\eps} \ll W \ll Q,
    \end{aligned}
  \end{equation*}
  and that~$\|f^{(j)}\|_\infty, \|\Psi^{(j)}\|_\infty \ll_j 1$. Then, if~$\eps>0$ is small enough in terms of~$\kappa$, we have
  $$ \ssum{w\in\N} \Psi\Big(\frac{w}{W}\Big) \ssum{n\in\N} \Lambda(n) f\Big(\frac{n}{X}\Big) \ud_R(n, bw) \ll Q^{1-\eps}\sqrt{X}. $$
  The implied constant depends at most on~$\kappa$, $A$, and the implied constants in the hypotheses.
\end{proposition}
\begin{proof} See Section \ref{se:prop}. \end{proof}

\subsection{Explicit formula}\label{section:explicit formula}

We let~$\kappa\in(0, \frac{50}{1093})$ be such that~$\supp\hat\phi \subset (-2-\kappa, 2+\kappa)$.

We rewrite the left-hand side of~\eqref{eq:main} by applying the explicit formula, \emph{e.g.}~\cite[Theorem~2.2]{Sica1998}, where the quantity $\Phi(\rho)$ there (not to be confused with our test function) is replaced by~$\phi(\frac{\rho-1/2}{2\pi i}\log Q)$, so that~$F(x) = \frac1{\log Q}{\hat \phi}(\frac x{\log Q})$. For~$q>1$ and~$\chi\mod{q}$ primitive, we obtain
\begin{equation}\label{eq:zerosum-grh}  
  \begin{aligned}
    \ssum{\rho\in\C \\ \Re(\rho)\in (0, 1) \\ L(\rho, \chi) = 0} \phi&\Big(\frac{(\rho-\frac12)\log Q}{2\pi i}\Big) \\
    = {}& O\Big(\frac1{\log Q}\Big) + {\hat \phi}(0)\frac{\log q}{\log Q} - \frac1{\log Q} \sum_{n\geq 1} (\chi(n) + \bar{\chi}(n)) \frac{\Lambda(n)}{\sqrt{n}} {\hat \phi}\Big(\frac{\log n}{\log Q}\Big),  
  \end{aligned}
\end{equation}
since the terms~$I, J$ appearing in~\cite[Theorem~2.2]{Sica1998} satisfy~$\abs{I(\frac12, b)} + \abs{J(\frac12, b)} \ll (\log Q)^{-1}$ for~$b\in\{0, \frac12\}$ by reasoning similarly as in~\cite[Lemma~3.1]{Sica1998}. Let $\Psi (x) = \Phi(x) x$. Summing \eqref{eq:zerosum-grh} over $\chi$ and $q$ we see that to conclude it remains to show that
\begin{equation}
  \label{eq:def-Skappa}
   S_\phi(Q) := \sum_{q\in\N} \frac1q \Psi\Big(\frac qQ\Big) \sum_{\substack{\chi (q) \\\text{primitive}}} \frac1{\log Q} \sum_{n\geq 1} (\chi(n) + \bar{\chi}(n)) \frac{\Lambda(n)}{\sqrt{n}} {\hat \phi}\Big(\frac{\log n}{\log Q}\Big) = o(Q).
\end{equation}
We will in fact obtain the following slightly stronger result. 
\begin{proposition}\label{prop:bound-sumprimes}
Let $\kappa\in(0,\frac{50}{1093})$. For all~$Q$ large enough and~$\eps>0$ small enough in terms of~$\kappa$, we have
$$ S_\phi(Q) = O\Big(\frac{Q}{\log Q}\Big). $$
The implied constant depends on~$\phi$ and~$\eps$ at most.
\end{proposition}

We break down the proof of Proposition \ref{prop:bound-sumprimes} into the following three sections. 

\subsection{Orthogonality and partition of unity}

Applying character orthogonality for primitive characters (see the third display in the proof of Lemma~4.1 of~\cite{BuiMilinovich2011}), we get
\begin{align}\label{eq:Skappa-orthog}
S_\phi(Q) = \frac{2}{\log Q} \summ{2}{v,w} \Psi \left( \frac{vw}{Q}\right) \frac{\mu(v)}v \frac{\varphi(w)}w \sum_{n \equiv 1 \mod{w}} \frac{\Lambda(n)}{\sqrt{n}}\hat{\phi}\left(\frac{\log n}{\log Q} \right).
\end{align}
Let~$V$ be any test function generating the partition of unity
$$ \sum_{j\in\Z} V\Big(\frac x{2^j}\Big) = 1 $$
for all~$x>0$. Inserting this in~\eqref{eq:Skappa-orthog}, we obtain
$$ S_\phi(Q) = \frac{2}{\log Q} \ssum{j\in\Z \\ 1/2 \leq X := 2^j \leq 2Q^{2+\kappa}} \summ{2}{v,w} \Psi \left( \frac{vw}{Q}\right) \frac{\mu(v)}v \frac{\varphi(w)}{w} \sum_{n \equiv 1 \mod{w}} \frac{\Lambda(n)}{\sqrt{n}} V\Big(\frac nX\Big) \hat{\phi}\left(\frac{\log n}{\log Q} \right). $$
Set~$f_{j}(x) = x^{-1/2}V(x) {\hat \phi}(\frac{\log(2^j x)}{\log Q})$ for~$\frac12 \leq 2^j \leq 2Q^{2+\kappa}$. Differentiating the product, we have that for all~$k\geq 0$, there exists~$C_{\phi,k}\geq 0$ such that~$\|f_j^{(k)}\|_\infty \leq C_{\phi,k}$ for all~$j$. We deduce
$$ S_\phi(Q) \ll   \sup_{1 \ll X \ll Q^{2+\kappa}} X^{-1/2} \sup_f \abs{T(Q, X)}, $$
where~$f$ varies among test functions subject to~$\|f^{(k)}\|_\infty \leq C_{\phi,k}$, and
$$ T(Q, X) :=  \summ{2}{v,w} \Psi \left( \frac{vw}{Q}\right) \frac{\mu(v)}v \frac{\varphi(w)}w \sum_{n \equiv 1 \mod{w}} \Lambda(n) f\Big(\frac nX\Big). $$

We handle the very small values of~$X$ by the trivial bound
$$ \sum_{n \equiv 1 \mod{w}} \Lambda(n) f\Big(\frac nX\Big) \ll \log Q \ssum{X/2 < n < 3X \\ n\neq 1, n\equiv 1\mod{w}} 1 \ll \frac{X\log Q}w, $$
which implies
$$ T(Q, X) \ll \frac{X \log Q}Q \summ{2}{vw \asymp Q} 1 \ll X (\log Q)^2. $$
It will therefore suffice to show that for
$$ Q^2 / (\log Q)^6 \ll X \ll Q^{2+\kappa}, $$
we have
$$ T(Q, X) \ll \frac{\sqrt{X} Q}{\log Q}. $$

\subsection{Substracting the main term}\label{sec:prel-reduct}

We insert the coprimality condition~$(n, v) = 1$. Since
\begin{align*}
{}& \summ{2}{v,w} \Psi \left( \frac{vw}{Q}\right) \frac{\mu(v)}v \frac{\varphi(w)}w \ssum{n \equiv 1 \pmod{w} \\ (n, v)>1} \Lambda(n) f\Big(\frac nX\Big) \\
\ll {}& \sum_{v\ll Q}v^{-1} \ssum{p\mid v \\ 1 \leq k \ll \log X} (\log p) \sum_{w\mid p^k-1} 1 \\
\ll {}& Q^{1+\eps},
\end{align*}
we obtain
$$ T(Q, X) = \summ{2}{v,w} \Psi \left( \frac{vw}{Q}\right) \frac{\mu(v)}v \frac{\varphi(w)}w \ssum{n \equiv 1 \pmod{w} \\ (n, v)=1} \Lambda(n) f\Big(\frac nX\Big) + O(Q^{1+\eps}). $$
Let~$1\leq R < Q/2$ so that~$R<vw$ for any~$v,w$ appearing in the sum. We replace the condition~$n\equiv 1\mod{w}$ by~$\ud_R(n, w)$. The difference is
$$ \sum_q \frac1q \Psi \left( \frac{q}{Q}\right) \ssum{\chi \mod{q} \\ r = \cond(\chi)\leq R \\ r \mid q} \ssum{(n, q)=1} \Lambda(n) f\Big(\frac nX\Big) \chi(n) \sum_{v\mid q/r}  \mu(v) = 0 $$
since~$r<q$ by our choice of~$R$, so that
$$ T(Q, X) = \summ{2}{v,w} \Psi \left( \frac{vw}{Q}\right) \frac{\mu(v)}v \frac{\varphi(w)}w \ssum{(n, v)=1} \Lambda(n) f\Big(\frac nX\Big) \ud_R(n, w) + O(Q^{1+\eps}). $$
We next remove the coprimality condition on~$n$, using the trivial bound~\eqref{eq:ud-triv}.
For the first term~$\1_{n\equiv 1\mod{w}}$ in~$\ud_R(n, w)$, this was already justified above. For the second term, we get
$$ \ll R Q^{-1+\eps} \summ{2}{v, w \\ vw \asymp Q} \sum_{p\mid v}\log p \ll R Q^{\eps}. $$
Since~$R\ll Q$, both error terms are acceptable. We get
$$ T(Q, X) = T(Q, X, R) + O(Q^{1+\eps}), $$
where
$$ T(Q, X, R) := \summ{2}{v,w} \Psi \left( \frac{vw}{Q}\right) \frac{\mu(v)}v \frac{\varphi(w)}w \Delta(w), $$
\begin{equation}
  \Delta(w) := \sum_{n} \Lambda(n) f\Big(\frac nX\Big) \ud_R(n, w).\label{eq:def-Delta}
\end{equation}
We are required to show that
\begin{equation}
T(Q, X, R) \ll \frac{\sqrt{X}Q}{\log Q}.\label{eq:lastgoal-bound}
\end{equation}

\subsection{Reduction to the critical range}

We now impose the additional conditions
\begin{equation}
Q^{\kappa/2 + \eps} \leq R \leq Q^{1/2}, \qquad \kappa < 2/3.\label{eq:conditions-LS}
\end{equation}
Observe that this $\kappa$ is the same as that appearing in the statement of Proposition \ref{prop:primes-ap}. The condition $\kappa < \frac{2}{3}$ is convenient for applying \eqref{eq:bound-delta-LS} below, but is rather loose since $\kappa$ is ultimately required to be much smaller than $\frac{2}{3}$.

Let~$B\in[1, Q^{1/2}]$ be a parameter. In~$T(Q, X, R)$, we write~$\frac{\varphi(w)}w = \sum_{b\mid w}\frac{\mu(b)}b$ and exchange summation, so that
\begin{align*}
  T(Q, X, R) \leq {}& \sum_{b, v} \frac1{bv}\Big| \sum_{w} \Psi\Big(\frac{bvw}{Q}\Big) \Delta(bw)\Big| \\
  \ll {}& (\log B)^2 \sup_{b, v \leq B}\Big| \sum_{w} \Psi\Big(\frac{bvw}{Q}\Big) \Delta(bw)\Big| + E_1 + E_2, \\
\end{align*}
where~$E_1$ (resp.~$E_2$) corresponds to the sum over~$b, v$ restricted to~$b>B$ (resp.~$v>B$). We recall that~$\supp \Psi \subset[\frac12, 3]$ by hypothesis. On the one hand, we have
\begin{align*}
  E_1 {}& \ll \summ{2}{b, w \\ bw \leq 3Q \\ b > B} \frac1b \abs{\Delta(bw)} \\
  {}& \ll Q^{\eps/2} B^{-1}\sum_{q\leq 3Q} \abs{\Delta(q)} \\
  {}& \ll Q^{1+\eps/2} \sqrt{X}B^{-1},
\end{align*}
using~\eqref{eq:bound-delta-LS} along with our hypotheses~\eqref{eq:conditions-LS}. On the other hand, we have
\begin{align*}
  E_2 {}& \ll \summ{2}{b, w \\ bw \leq 3Q/B} \frac1b \abs{\Delta(bw)} \\
  {}& \ll Q^{\eps/2} \sum_{q\leq 3Q/B} \abs{\Delta(q)} \\
  {}& \ll Q \sqrt{X} (Q^{\eps/2} B^{-1} + Q^{-\eps})
\end{align*}
again by~\eqref{eq:conditions-LS} and~\eqref{eq:bound-delta-LS}; we have used the bounds~$Q^{-1+\eps}\sqrt{X} R^{-1} \ll Q^{-\eps}$ and~$Q^{-1+\eps} X^{3/8} \ll Q^{-\eps}$, which follow from~$Q^{\kappa/2+\eps} \leq R$ and~$\kappa < 2/3$ respectively upon reinterpreting~$\eps$.

Grouping the above, it will suffice to show that
$$ \sum_{w} \Psi\Big(\frac{bvw}{Q}\Big) \Delta(bw) \ll Q^{1-\eps}\sqrt{X} $$
uniformly for~$b, v \leq Q^\eps$ and test functions~$\Psi$ and~$f$. Assume now~$\kappa\in(0, \frac{50}{1093})$. Then the conditions on~$R$ in~\eqref{eq:conditions-LS} and in Proposition~4 overlap, so that we may apply Proposition~\ref{prop:primes-ap} with~$W = \frac{Q}{bv}$. This gives the above bound, and completes the proof of~\eqref{eq:lastgoal-bound}, hence of Proposition~\ref{prop:bound-sumprimes}.

\section{Exponential sums estimates}

In this section, we work out the modifications to be made to the arguments underlying~\cite{DeshouillersIwaniec1982a} in order to exploit current knowledge on the spectral gap of the Laplacian on congruence surfaces~\cite{KimSarnak2003}. We will follow the setting in Theorem~2.1 of~\cite{Drappeau2017}, since we will need to keep track of the uniformity in~$q_0$. We also take the opportunity to implement the correction recently described in~\cite{BombieriEtAl2019}.

Let~$\theta\geq 0$ be a bound towards the Petersson-Ramanujan conjecture, in the sense of~\cite[eq.~(4.6)]{Drappeau2017}. Selberg's $3/16$ theorem corresponds to~$\theta \leq 1/4$, and the Kim-Sarnak bound~\cite{KimSarnak2003} asserts that~$\theta\leq 7/64$.

\begin{proposition}\label{prop:newbound-DI}
  Let the notations and hypotheses be as in~\cite[Theorem~2.1]{Drappeau2017}. Then
  \begin{align*}
    & \underset{\substack{c\equiv c_0 \text{ and } d \equiv d_0 \mod{q}\\(qrd, sc) = 1}}{\sum_c \sum_d \sum_n \sum_r \sum_s} b_{n,r,s} g(c,d,n,r,s) \e\Big(n\frac{\bar{rd}}{sc}\Big) \\
    & \qquad \ll_{\eps, \eps_0} (qCDNRS)^{\eps+O(\eps_0)} q^{3/2} K(C, D, N, R, S) \|b_{N,R,S}\|_2,
  \end{align*}
  where~$\|b_{N,R,S}\|_2^2 = \sum_{n,r,s} \abs{b_{n,r,s}}^2$, and here
  \begin{equation}
    \begin{aligned}
      K(C, D, N, R, S)^2 ={}& qCS(RS+N)(C+RD) \\
      &{} + C^{1+4\theta} DS ((RS+N)R)^{1-2\theta}\red{(1 + \tfrac{qC}{RD})^{1-4\theta}} + D^2NR.
    \end{aligned}\label{eq:def-DI-K}
  \end{equation}
\end{proposition}

\begin{remark}
  The bound of Proposition~\ref{prop:newbound-DI} is monotonically stronger as~$\theta$ decreases, since the first term is larger \red{than~$CS(RS+N)(RD+qC)$}. Under the Petersson-\red{Ramanujan} conjecture for Maass forms, which predicts that~$\theta=0$ is admissible, the second term in~\eqref{eq:def-DI-K} is smaller than the first.
\end{remark}

\begin{proof}
  The proof of the proposition, as with all results of this type, relies on the Kuznetsov formula and large sieve inequalities for coefficients of automorphic forms. The application of the Kuznetsov formula requires one to understand the contribution of holomorphic forms, Eisenstein series, and Maass forms (whether the holomorphic forms appear depends on the sign of the variables inside the Kloosterman sum). We divide these forms into the \emph{exceptional spectrum} and the \emph{regular spectrum}. The exceptional spectrum consists of those (conjecturally non-existent) Maass forms whose eigenvalues $t_f = \frac{1}{2}+it_f$ have $t_f \in i \mathbb{R}$. By the definition of $\theta$ above we have that $|t_f|\leq \theta$ for all $f$ in the exceptional spectrum. The regular spectrum consists of everything that is not exceptional. The contribution of the regular spectrum is handled as in \cite{Drappeau2017}, and does not require any modification here. We improve upon the analysis of \cite{Drappeau2017} in handling the exceptional spectrum by keeping track of the dependence on $\theta$ (see the remark made in \cite[p. 703]{Drappeau2017}). The statements of \cite{Drappeau2017} which are affected are Lemma~4.10, Proposition~4.12, Proposition 4.13 and the proof of Theorem~2.1. The treatment of the exceptional spectrum rests upon a weighted large sieve inequality. These weighted large sieve inequalities are proved, following \cite{DeshouillersIwaniec1982a}, by an iterative procedure.

  With the notations of~\cite{Drappeau2017}, the changes to be made are as follows~:
  \begin{itemize}
  \item Lemma~4.10 bounds sums of the form
\begin{align*}
\sum_{\substack{q\leq Q \\ q_0 \mid q}} \sum_{\substack{f \in \mathcal{B}(q,\chi) \\ t_f \in i\mathbb{R}}} Y^{2|t_f|} \left|\sum_{N < n \leq 2N} n^{1/2} \rho_{f \infty}(n) \right|^2,
\end{align*}  
and serves to control the first step of the recursion. The bound
\begin{align*}
\sum_{\substack{q\leq Q \\ q_0 \mid q}} \sum_{\substack{f \in \mathcal{B}(q,\chi) \\ t_f \in i\mathbb{R}}} Y^{2|t_f|} \left|\sum_{N < n \leq 2N} n^{1/2} \rho_{f \infty}(n) \right|^2 &\ll (QN)^\varepsilon (Q q_0^{-1} + N + (NY)^{1/2}) N
\end{align*}
may be replaced by the bound
\begin{align*}
&\ll (QN)^\varepsilon \big(Q q_0^{-1} + N + (NY)^{2\theta}(Q^{1-4\theta} + N^{1-4\theta})\big)N.
\end{align*}
This does not require any change in the recursion argument, but merely to use the bound~$\abs{t_f} \leq \theta$ in the very last step, page 278 of~\cite{DeshouillersIwaniec1982a}, whereby~$\sqrt{Y/Y_1}$ is replaced by~$(Y/Y_1)^{2\theta}$.
  \item In Proposition~4.12 one bounds sums of the form
\begin{align*}
\sum_{\substack{m,n,r,s \\ (s,rq)=1}} a_m b_{n,r,s}\sum_{c \in \mathcal{C}(\infty,1/s)} \frac{1}{c}\phi\left(\frac{4\pi \sqrt{mn}}{c} \right) S_{\infty,1/s}(m,\pm n;c)
\end{align*}  
in terms of quantities $L_{\text{reg}}$ and $L_{\text{exc}}$. In place of 
\begin{align*}
L_{\text{exc}} = \left(1 + \sqrt{\frac{N}{RS}} \right) \sqrt{\frac{1+X^{-1}}{RS}} \left(\frac{MN}{RS+N} \right)^{1/4}\frac{\sqrt{RS}}{1+X} \sqrt{M}\| b_{N,R,S}\|_2
\end{align*}
we claim the improved
\begin{align*}
L_{\text{exc}} = q_0^{\frac12-2\theta} \Big(1 + \sqrt{\frac{N}{RS}}\Big) \Big(\frac{1+X^{-1}}{RS}\Big)^{2\theta} \Big(\frac{MN}{RS+N}\Big)^\theta \red{\Big(1 + \frac{M}{RS}\Big)^{1/2-2\theta}} \frac{\sqrt{RS}}{1+X} \sqrt{M} \|b_{N,R,S}\|_2.
\end{align*}
To obtain this bound one uses the new bound for Lemma~4.10 and follows the arguments of \cite[section 9.1]{DeshouillersIwaniec1982a}.
  \item In Proposition~4.13, one bounds
\begin{align*}
\sum_{\substack{c,m,n,r,s \\ (sc,rq)=1}}b_{n,r,s}\overline{\chi}(c)g(c,m,n,r,s)e(mt) S(n\overline{r},\pm m\overline{q};sc)
\end{align*}  
in terms of quantities $K_{\text{reg}}$ and $K_{\text{exc}}$. The term
\begin{align*}
K_{\text{exc}}^2 = C^3S^2 \sqrt{R(N+RS)}
\end{align*}
can be replaced by
\begin{align*}
K_{\text{exc}}^2 = C^{2+4\theta} S^2 (R(N+RS))^{1-2\theta} \red{(1 + \tfrac{M}{RS})^{1-4\theta}}. 
\end{align*}
This is seen by using the new definition on~$L_{\text{exc}}$ in Proposition~4.12, and by keeping track of a factor~$q^{-1+2\theta}$ coming from the term~$(1+X^{-1})^{2\theta} / (1+X)$.
  \item Finally, we modify the proof of Theorem~2.1 at two places. First, the bound for~$\mathcal{A}_0$ on page~706, as explained in~\cite{BombieriEtAl2019}, is wrong unless further hypotheses on~$(b_{n,r,s})$ are imposed. The correct bound in general is
    $$ \mathcal{A}_0 \ll q^{-2} (\log S)^2 D (NR)^{1/2} \|b_{N,R,S}\|_2, $$
    and this yields the term~$D^2NR$ instead of~$D^2NR S^{-1}$. Secondly, our new bound for~$K_{\text{exc}}$ in Proposition~4.13 gives a contribution~$C^{2+4\theta} S^2 (R(RS+N))^{1-2\theta} \red{(1 + \tfrac{M_1}{RS})^{1-4\theta}}$ instead of~$C^3 S^2 \sqrt{R(RS+N)}$ in the definition of~$L_{\text{exc}}^2$ and~$L^*(M_1)^2$ on p.707 of~\cite{Drappeau2017}. This yields a term~$C^{1+4\theta}DS((N+RS)R)^{1-2\theta}\red{(1 + \tfrac{qC}{RD})^{1-4\theta}}$ instead of~$C^2 DS \sqrt{(N+RS)R}$ in eq.~(4.39) of~\cite{Drappeau2017}, and by following the rest of the arguments we deduce our claimed bound.
  \end{itemize}
\end{proof}

\section{Primes in arithmetic progressions: Proof of Proposition \ref{prop:primes-ap}} \label{se:prop}

The proof of Theorem~\ref{th:main} relies on Proposition \ref{prop:primes-ap} which for the convenience of the reader we recall below. 

\begin{proposition*}
  Let~$\kappa \in (0, \frac{50}{1093})$ and~$\eps>0$. Let~$\Psi, f$ be test functions, $A>0$, $X, Q, W, R\geq 1$, and~$b\in\N$. Assume that
  \begin{equation*}
    \begin{aligned}
      \frac{Q^2}{(\log Q)^A} \ll X \ll Q^{2+\kappa}, \quad & \quad X^{11/20} Q^{-1} \leq R\leq Q^{2/3} X^{-2/9}, \\
      b \leq Q^\eps \quad & \quad Q^{1-\eps} \ll W \ll Q,
    \end{aligned}
  \end{equation*}
  and that~$\|f^{(j)}\|_\infty, \|\Psi^{(j)}\|_\infty \ll_j 1$. Then, if~$\eps>0$ is small enough in terms of~$\kappa$, we have
  $$ \ssum{w\in\N} \Psi\Big(\frac{w}{W}\Big) \ssum{n\in\N} \Lambda(n) f\Big(\frac{n}{X}\Big) \ud_R(n, bw) \ll Q^{1-\eps}\sqrt{X}. $$
  The implied constant depends at most on~$\kappa$, $A$, and the implied constants in the hypotheses.
\end{proposition*}

\begin{remark}
What is crucial in our statement is the size of the upper bound, which should be negligible with respect to~$Q\sqrt{X}$. On the other hand, we are only interested in values of~$X$ larger than~$Q^2$. This is in contrast with most works on primes in arithmetic progressions~\cite{FouvryIwaniec1983,BombieriFriedlanderEtAl1986,Zhang2014}, where the main challenge is to work with values of~$X$ much smaller than~$Q^2$, while only aiming at an error term which is negligible with respect to~$X$. The main point is that in both cases, the large sieve yields an error term which is always too large (see~\cite[Theorem~17.4]{IwaniecKowalski2004}), an obstacle which the dispersion method is designed to handle.
\end{remark}

In what follows, we will systematically write
$$ X = Q^{2+\varpi}, $$
so that~$-o(1) \leq \varpi \leq \kappa + o(1)$ as~$Q\to\infty$.

\subsection{Combinatorial identity}

We perform a combinatorial decomposition of the von Mangoldt function into sums of different shapes: Type $d_1$ sums have a long smooth variable, Type $d_2$ sums have two long smooth variables, and Type II sums have two rough variables that are neither too small nor too large. We accomplish this decomposition with the Heath-Brown identity and the following combinatorial lemma.

\begin{lemma}\label{lemma:combinatorial lemma}
Let $\{t_j\}_{1 \leq j \leq J} \in \mathbb{R}$ be non-negative real numbers such that $\sum_j t_j = 1$. Let $\lambda,\sigma,\delta \geq 0$ be real numbers such that
\begin{itemize}
\item $\delta < \tfrac{1}{12}$,
\item $\sigma \leq \tfrac{1}{6} - \tfrac{\delta}{2}$,
\item $2\lambda + \sigma < \frac{1}{3}$.
\end{itemize}
Then at least one of the following must occur:
\begin{itemize}
\item (Type $d_1$) There exists $t_j$ with $t_j \geq \tfrac{1}{3} + \lambda$.
\item (Type $d_2$) There exist $i,j,k$ such that $\tfrac{1}{3}-\delta < t_i,t_j,t_k < \tfrac{1}{3} + \lambda$, and
\begin{align*}
\sum_{t_j^* \not \in \{t_i,t_j,t_k\}} t_j^* < \sigma.
\end{align*}
\item (Type II) There exists $S\subset \{1,\ldots,J\}$ such that
\begin{align*}
\sigma \leq \sum_{j \in S} t_j \leq \tfrac{1}{3} - \delta.
\end{align*}
\end{itemize}
\end{lemma}
\begin{proof}
Assume that the Type $d_1$ case and the Type II case both fail. Then for every $j$ we have $t_j < \tfrac{1}{3} + \lambda$, and for every subset $S$ of $\{1,\ldots,J\}$ we either have
\begin{align*}
\sum_{j \in S} t_j < \sigma
\end{align*}
or
\begin{align*}
\sum_{j \in S} t_j > \frac{1}{3} - \delta.
\end{align*}
Let $s_1,\ldots,s_K$ denote those $t_j$ with $\tfrac{1}{3} - \delta < t_j < \tfrac{1}{3} + \lambda$. We will show that $K = 3$. Let $t_j^*$ be any other $t_j$, so that $t_j^* \leq \tfrac{1}{3} - \delta$, and therefore $t_j^* < \sigma$. We claim that
\begin{align*}
\sum_j t_j^* < \sigma.
\end{align*}
If not, then $\sum_j t_j^* > \tfrac{1}{3} - \delta$. By a greedy algorithm we can find some subcollection $S^*$ of the $t_j^*$ such that
\begin{align*}
\sigma < \sum_{j \in S^*} t_j^* \leq 2\sigma.
\end{align*}
Since $2\sigma \leq \tfrac{1}{3} - \delta$ this subcollection satisfies the Type II condition, in contradiction to our assumption.

Now we show that $K = 3$. Observe that $K \geq 3$, since if $K \leq 2$ we have
\begin{align*}
1 = \sum_j t_j = \sum_{i=1}^K s_i + \sum_j t_j^* < 2 \left(\tfrac{1}{3} + \lambda\right) + \sigma < 1.
\end{align*}
Furthermore, we must have $K \leq 3$, since if $K \geq 4$ we have
\begin{align*}
1 = \sum_j t_j \geq \sum_{i=1}^K s_i > 4 \left(\frac{1}{3} - \delta \right) > 1.
\end{align*}
This completes the proof.
\end{proof}

Using \emph{e.g.} Heath-Brown's combinatorial identity~\cite{Heath-Brown1982}, we deduce the following.

\begin{corollary}\label{cor:after-hb}
  Let~$f$ be a test function, $u:\N\to\C$ be any map, and~$X\geq 1$.
  Then there exists a sequence~$(C_j)_{j\geq 0}$ of positive numbers, depending only on~$f$, such that we have
  \begin{equation}
    \abs{\ssum{n\in\N} \Lambda(n) f\Big(\frac{n}{X}\Big) u(n)} \ll (\log X)^{8} (T_1 + T_2 + T_{\text{II}}), \label{eq:bound-hb}
  \end{equation}
  where
  \begin{align}
    T_1 ={}&  \sup_{\substack{N \gg X^{1/3+\lambda} \\ MN \asymp X}} \sup_{\substack{g\in {\mathcal G} \\ \beta \in {\mathcal S}}} \abs{\summ{2}{n\in\N \\ m\sim M} g\Big(\frac nN\Big) \beta_m u(mn)} , \label{eq:T1} \\
    T_2 = {}& \sup_{\substack{X^{1/3-\delta} \ll N_2 \leq N_1 \ll X^{1/3+\lambda} \\ MN_1N_2 \asymp X}} \sup_{\substack{g_1, g_2 \in \mathcal{G} \\ \beta \in {\mathcal S}}} \abs{\summ{3}{n_1, n_2 \in\N \\ m \sim M} g_1\Big(\frac {n_1}{N_1}\Big) g_2\Big(\frac {n_2}{N_2}\Big) \beta_m u(mn_1n_2) }, \label{eq:T2} \\
    T_{\text{II}} = {}& \sup_{\substack{X^{\sigma} \ll N \ll X^{1/3-\delta} \\ MN\asymp X}} \sup_{\alpha, \beta \in \mathcal{S}} \abs{\summ{2}{n \sim N \\ m\sim M} \alpha_{m} \beta_n u(mn) }, \label{eq:TII}
  \end{align}
  where the implied constants are absolute, $\mathcal{G}$ is the set of test functions~$g$ satisfying~$\|g^{(j)}\|_\infty \leq C_j$ and~$\mathcal{S}$ is the set of sequences~$(\beta_n)$ satisfying~$\abs{\beta_n} \leq d(n)^{8}$.
\end{corollary}

\begin{proof}
  By the Heath-Brown identity~\cite[Proposition~13.3]{IwaniecKowalski2004}, there exists bounded coefficients~$(c_J)_{1\leq J\leq 4}$ such that
  $$ \Lambda(n) = \sum_{J=1}^4 c_J \ssum{m_1, \dotsc, m_J \\ n_1, \dotsc, n_J \\ n= m_1 \dotsc m_J n_1 \dotsc n_J \\ m_j \leq (3X)^{1/4}} \log(n_1) \prod_j \mu(m_j) $$
  for any~$n$ involved in the left-hand side of~\eqref{eq:bound-hb}. Let~$\psi$ be a test function inducing a partition of unity in the sense that~$\sum_{j\in\Z} \psi(\frac x{2^j}) = 1$ for all~$x>0$. Then we have
  $$ \sum_{n\in\N} \Lambda(n) f\Big(\frac nX\Big) u(n) = \sum_{J=1}^4 c_J \sum_{(M_1, \dots, M_J, N_1, \dotsc, N_J) \in U_J} S(M_1, \dotsc, M_J, N_1, \dotsc, N_J), $$
  $$ S(M_1, \dotsc, N_J) = \sum_{m_1, \dotsc, n_J \in \N} \log(n_1) \Big(\prod_j \psi\Big(\frac{n_j}{N_j}\Big)\Big) \Big(\prod_j \mu^*(m_j)\psi\Big(\frac{m_j}{M_j}\Big)\Big) f\Big(\frac{m_1 \dotsc n_J}{X}\Big) u(m_1 \dotsc n_J), $$
  where~$U_J$ is the set of~$2J$-tuples of powers of~$2$ such that~$X/6 \leq M_1 \dotsc M_J N_1 \dotsc N_J \leq 6X$, and~$\mu^*(m) = \mu(m)$ if~$m\leq (3X)^{1/4}$ and~$0$ otherwise. We abbreviated~$m_1\dotsc n_J = m_1\dotsc m_J n_1 \dotsc n_J$. The set~$U_J$ has at most~$O((\log X)^{2J-1})$ elements. By Lemma~\ref{lemma:combinatorial lemma}, for each choice of~$J$ and~$(M_1, \dotsc, N_J) \in U_J$ we have either~$N \geq \frac16 X^{1/3+\lambda}$ for some~$N\in\{N_j\}$, or~$\frac16X^{1/3-\delta} \leq N', N'' \leq 6X^{1/3+\lambda}$ for some~$N', N'' \in \{N_j\}$, or~$\frac16X^{\sigma} \leq N \leq 6X^{1/3-\delta}$ for some subproduct~$N$ of~$N_j$ and~$M_j$ (here we used that for~$X$ large enough, we have~$(3X)^{1/4} < \frac16 X^{1/3-\delta}$). Sorting the sum over~$J$ and~$(M_1, \dotsc, N_J)$ according to this trichotomy, and writing~$\log(n_1) = \log N_1 + \log(n_1/N_1)$, the above is bounded in absolute values by
  $$ \ll (\log X)^8 (T_1^* + T_2^* + T_{\text{II}}^*), $$
  \begin{align*}
    T_1^* = {}& \sup_{\substack{X/6\leq MN \leq 6X \\ \frac16 X^{1/3+\lambda} \leq N \\ |r| \leq 8}} \sup_{\substack{g\in \{\psi, \psi \log\} \\ \beta \in \mathcal{S}}} \abs{\ssum{n\in \N \\ m\sim M} g\Big(\frac nN\Big) \beta_m f\Big(\frac{mn}{2^r MN}\Big) u(mn)}, \\
    T_2^* = {}& \sup_{\substack{X/6 \leq N_1N_2M \leq 6X \\ \frac16 X^{1/3-\delta} \leq N_1, N_2 \leq 6X^{1/3+\lambda} \\ |r|\leq 8}} \sup_{\substack{g_1, g_2 \in \{\psi, \psi\log\} \\ \beta \in \mathcal{S}}} \abs{\ssum{n_1, n_2 \in \N \\ m\sim M} g_1\Big(\frac{n_1}{N_1}\Big)g_2\Big(\frac{n_2}{N_2}\Big) \beta_m f\Big(\frac{n_1n_2m}{2^rN_1N_2M}\Big) u(n_1n_2m)}, \\
    T_{\text{II}}^* = {}& \sup_{\substack{X/6 \leq NM \leq 6X \\ \frac16X^\sigma \leq N \leq 6X^{1/3-\delta} \\ |r|\leq 8}} \sup_{\alpha, \beta \in \mathcal{S}} \abs{\ssum{m\sim M \\ n\sim N} \alpha_m \beta_n f\Big(\frac{mn}{2^rMN}\Big) u(mn)}.
  \end{align*}
  Here the conditions~$m\sim M$ and~$n\sim N$ in the sums were added by an additional bounded dichotomy (which is the reason for the presence of the sup over~$r$).
  Finally, letting~${\check f}$ be the Mellin transform of~$f$, we have by Mellin inversion~$f(x) = \frac1{2\pi} \int_{-\infty}^\infty {\check f}(it) x^{-it} \df t$, and the map~$t\mapsto {\check f}(it)$ is of Schwartz class on~$\R$. In particular, for~$M, N, r, g, \beta$ as in~$T_1^*$ we have
  $$ \abs{\ssum{n\in \N, m\sim M} g\Big(\frac nN\Big) \beta_m f\Big(\frac{mn}{2^rMN}\Big) u(mn)} \ll \sup_{t\in\R} \abs{\ssum{n\in \N, m\sim M} g_t\Big(\frac nN\Big) \beta_{m,t} u(mn)} $$
  where~$g_t(x) = (1+t^2){\check f}(it) x^{-it} g(x)$ (the factor~$1+t^2$ being included so that we could write a supremum) and~$\beta_{m,t}=m^{-it}\beta_m \in\mathcal{S}$. We note that~$g_t$ is a test function satisfying~$\|g_t^{(j)}\|_\infty \ll C_j$ where~$C_j := \sup_{0\leq k, \ell, m \leq j+2} \|t^k{\check f}(it)\|_\infty\|x^{-\ell} g^{(m)}(x)\|_\infty$ can be bounded in terms of~$f$ only. This yields the contribution of~$T_1$ in our claim. The contributions of~$T_2$ and~$T_{\text{II}}$ are obtained in the same way.
\end{proof}

In what follows, we successively consider~$T_1, T_2$ and~$T_{\text{II}}$, which we specialize at
$$ u(n) :=  \sum_{w\in\N} \Psi\Big(\frac{w}{W}\Big) \ud_R(n, bw), $$
and we will denote
$$ R = X^\rho. $$

\subsection{Type~$d_1$ sums}

We suppose~$M$ and~$N$ are given as in~\eqref{eq:T1}. The quantity we wish to bound is
\begin{align*}
T_1(M, N) = \sum_{w} \Psi\Big(\frac wW\Big) \ssum{m \sim M \\ (m, bw)=1} \beta_m\Big(&\ssum{n\in\N \\ mn \equiv 1\mod{bw}} g\Big(\frac nN\Big) \\
{}& - \frac1{\varphi(bw)} \ssum{\chi\mod{bw}\\\cond(\chi)\leq R} \chi(m) \sum_{(n,bw)=1} \chi(n) g\Big( \frac{n}{N}\Big) \Big).
\end{align*}
By Poisson summation and the classical bound on Gauss sums~\cite[Lemma~3.2]{IwaniecKowalski2004}, we have
\begin{align*}
\sum_{n \equiv \overline{m} \mod{bw}} g \left(\frac{n}{N} \right) &= \frac{N}{bw}\hat{g}(0) + \frac{N}{bw}\sum_{0 <|h| \leq W^{1+\eps}/N} \hat{g}\left(\frac{Nh}{bw} \right)e \left(\frac{\overline{m}h}{bw} \right) + O \left(Q^{-A} \right), \\
\frac{1}{\varphi(bw)}\sum_{(c,bw)=1} \chi(c) g \left(\frac{c}{N} \right) &= \frac{N}{bw}\hat{g}(0)\1(\chi=\chi_0) + O\left(\frac{Q^\eps R^{1/2}}{W}\right).
\end{align*}
Therefore,
\begin{align*}
T_1(M, N) = \frac{N}b\sum_{w} \frac1w \Psi\left( \frac{w}{W}\right) \sum_{\substack{(m,bw)=1 \\ m \sim M}} \beta_m \sum_{0 <|h| \leq W^{1+\eps}/N} \hat{g}\left(\frac{Nh}{bw} \right) e\left(\frac{\overline{m}h}{bw} \right) + O(M R^{3/2} Q^{\eps}).
\end{align*}
Our goal is to get cancellation in the exponential phases by summing over the smooth variable $w$. We apply the reciprocity formula
\begin{align*}
\frac{\overline{m}h}{bw} \equiv - \frac{\overline{bw}h}{m} + \frac{h}{mbw} \ (\text{mod }1),
\end{align*}
which implies
\begin{align*}
T_1(M, N) = {}&\frac Nb\sum_{w} \frac1w \Psi\left( \frac{w}{W}\right) \sum_{\substack{(m,bw)=1 \\ m \sim M}} \beta_m \sum_{0 <|h| \leq W^{1+\eps}/N} \hat{g}\left(\frac{Nh}{bw} \right) e\left(\frac{\overline{bw}h}{m} \right) \\
& \qquad + O(M R^{3/2} Q^{\eps} + Q^{1+\eps}N^{-1}).
\end{align*}
We rearrange the sum as
\begin{align*}
\frac N{bW} \sum_{\substack{(m,b)=1 \\ m \sim M}} \beta_m \sum_{0 <|h| \leq W^{1+\eps}/N} \ \sum_{(w,m)=1}\hat{g}\left(\frac{Nh}{bw} \right) \frac Ww \Psi \left( \frac{w}{W}\right) e \left(\frac{\overline{bw}h}{m} \right).
\end{align*}
By partial summation and a variant of the Weil bound~\cite[eq.~(2.4)]{Drappeau2015}, the sum on $w$ is
\begin{align*}
\ll ((h, m)WM^{-1} + \sqrt{(h,m)} \sqrt{M}) Q^\eps.
\end{align*}
Summing over $h$ and $m$, we obtain a bound
\begin{align*}
T_1(M, N) \ll Q^{1+\eps} + M^{3/2} Q^\eps + M R^{3/2} Q^\eps.
\end{align*}
This bound is acceptably small provided
\begin{align*}
N \gg{}& \Big(\frac{X}{Q}\Big)^{2/3+\eps} = X^{\frac{1}{3} + \frac{\varpi}{3(2+\varpi)}+\frac{1+\varpi}{2+\varpi}\eps}, \\ 
N \gg{}& \frac{X^{1/2}R^{3/2}}{Q^{1-2\eps}} = X^{\frac{\varpi}{2(2+\varpi)} + \frac32\rho + \frac{2\eps}{2+\varpi}}.
\end{align*}
These inequalities are satisfied, for all sufficiently small~$\eps>0$, under the assumptions
\begin{equation}
\lambda > \frac{\varpi}{3(2+\varpi)}, \qquad \rho < \frac{4+\varpi}{9(2+\varpi)}.\label{eq:hyp-T1}
\end{equation}

We have proved the following.
\begin{lemma}\label{lem:bd-T1}
  Under the notations and hypotheses of Corollary~\ref{cor:after-hb}, and assuming~\eqref{eq:hyp-T1}, we have
  $$ T_1 \ll Q^{1-\eps}\sqrt{X}. $$
  The implied constant depends on~$\lambda, \rho$ and~$\varpi$.
\end{lemma}

\subsection{Type~$d_2$ sums}

The treatment of the type~$d_2$ sums~\eqref{eq:T2} is nearly identical to~\cite[Section~14]{BombieriFriedlanderEtAl1986}. For convenience, we rename~$(N_1, N_2, M)$ into~$(M, N, L)$ so that we have~$MNL\asymp X$. We wish to bound
\begin{align*}
  T_2(M, N, L) = \sum_{\ell \sim L} \beta_\ell \sum_{(w,\ell)=1} & \Psi\left(\frac{w}{W} \right)\bigg(\summ{2}{m,n \\ \ell mn \equiv 1\mod{bw}} g_1\left(\frac{m}{M}\right) g_2 \left( \frac{n}{N}\right) \\ 
    & \qquad - \frac{1}{\varphi(bw)} \ssum{\chi\mod{bw} \\ \cond(\chi)\leq R} \chi(\ell) \summ{2}{(mn,bw)=1} g_1\left(\frac{m}{M} \right) g_2\left(\frac{n}{N} \right)  \chi(mn) \bigg).
\end{align*}
We perform Poisson summation on the $m$-sums to get
\begin{align*}
\sum_{m \equiv \overline{\ell n} \mod{bw}} g_1\left(\frac{m}{M} \right) &= \frac{M}{bw} \sum_{|h| \leq H} \hat{g_1}\left(\frac{Mh}{bw} \right) e\left(\frac{\overline{\ell n}h}{bw} \right) + O(Q^{-A}), \\
\sum_{(m,bw)=1} \chi(m) g_1\left(\frac{m}{M} \right) &= \frac{\varphi(bw)}{bw}M\hat{g_1}(0)\1(\chi=\chi_0) + O \left(Q^{\eps}R^{1/2} \right),
\end{align*}
where $H = W^{1+\eps} M^{-1}$. The contribution of the error terms is
\begin{align*}
\ll LN R^{3/2} Q^\eps.
\end{align*}
The zero frequency of Poisson summation cancels out. For the non-zero frequencies we employ reciprocity in the form
\begin{align*}
e \left(\frac{\overline{\ell n}h}{bw} \right) =e \left(-\frac{\overline{bw}h}{\ell n} \right) + O \left(\frac{H}{LNW} \right),
\end{align*}
and the error term contributes a quantity of size $O(Q^{1+\eps})$. We therefore have
\begin{equation}\label{eq:T2-after-poisson}
\begin{aligned}
T_2(M, N, L) = {}& \frac{M}b \ssum{\ell \sim L \\ (\ell, b) = 1} \beta_\ell \sum_{\substack{(w,\ell)=1}} \frac1w \Psi \left(\frac{w}{W} \right)\sum_{\substack{(n,bw)=1 }}g_2\left(\frac{n}{N} \right) \sum_{0 <|h| \leq H} \hat{g_1} \left(\frac{Mh}{bw} \right)e \left(-\frac{\overline{bw}h}{\ell n} \right) \\
& \qquad + O(Q^{1+\eps} + LNR^{3/2}Q^\eps).
\end{aligned}
\end{equation}

We next separate the variables $h$ and $w$. We change variables to write
\begin{align*}
\hat{g_1} \left(\frac{Mh}{bw} \right) &= \frac{w}{M}\int_\R g_1\left(\frac{wy}{M} \right) e\left(-\frac{hy}{b} \right)dy.
\end{align*}
Since~$g_1$ and~$\Psi$ are test functions, the integral is restricted to $y \asymp M/W$. We move the integral to the outside to write the first term of the right-hand side of~\eqref{eq:T2-after-poisson} as
\begin{equation}\label{eq:T2-expsum}
  \ll \frac{M}{bW}\sup_{y \asymp M/W}\abs{ \sum_\ell \beta_\ell \sum_{0<\abs{h}\leq H} \e\left(-\frac{hy}{b} \right) \sum_w \sum_n \Psi\left(\frac{w}{W} \right) g_1\left(\frac{wy}{M} \right) g_2\left(\frac{n}{N} \right) \e\left(-\frac{\overline{bw}h}{\ell n} \right)}.
\end{equation}
We then use \cite[Theorem~12]{DeshouillersIwaniec1982a}, amended as described in~\cite{BombieriEtAl2019}, more specifically, with the dictionary (the bold symbols denote the variables names from~\cite{DeshouillersIwaniec1982a})
\begin{align*}
  {\bm{c}, \bm{C}} \leftrightarrow n, N, & \qquad {\bm{d}, \bm{D}} \leftrightarrow w, W, \\
  {\bm{n}, \bm{N}} \leftrightarrow h, H, &\qquad {\bm{r}, \bm{R}} \leftrightarrow b', b, \\
  {\bm{s}, \bm{S}} \leftrightarrow \ell, L, &\qquad {\bm{b}}_{\bm{n}, \bm{r}, \bm{s}} \leftrightarrow \1_{b'=b} \e(-hy/b)\beta_\ell.
\end{align*}
Since~$\lambda<1/6$, we have~$H \ll L$ if~$\eps$ is sufficiently small. Therefore, with the same notations, we find the bounds
$$ {\bm{K}}(\bm{C}, \bm{D}, \bm{N}, \bm{R}, \bm{S}) \ll b(NL^2(N+W) + N^2WL^{3/2} + W^2H)^{1/2}, $$
$$ \|{\bm b}_{\bm N, \bm R, \bm S}\|_2 \ll L^\eps (HL)^{1/2}. $$
It will also be easier to sum up the bounds if we assume
\begin{equation}
N \ll W^{1+\eps}.\label{eq:T2-hyp1}
\end{equation}
We find
\begin{align*}
T_2(M, N, L) &\ll LNR^{3/2} Q^\eps + Q^\eps\left(\sqrt{X}L + \sqrt{M}NL^{5/4} + L^{1/2}W\right) \\
&\ll LNR^{3/2} Q^\eps + Q^\eps\left(\sqrt{X}L + \sqrt{M}NL^{5/4}\right),
\end{align*}
the second inequality following since $L^{1/2}W \ll X^{1/2} L$.
This contribution is acceptable provided
\begin{equation}
M\gg X^{\frac{\varpi}{2(2+\varpi)} + \frac32\rho + \eps}, \qquad MN \gg X^{\frac{1}{2}+\frac{\varpi}{2(2+\varpi)}+\eps}\label{eq:T2-hyp2}
\end{equation}
and
\begin{equation}
M^{3/2}N^{1/2} \gg X^{\frac{1}{2} + \frac{\varpi}{2+\varpi}+2\eps}.\label{eq:T2-hyp3}
\end{equation}
The bounds~\eqref{eq:T2-hyp1}--\eqref{eq:T2-hyp3} are satisfied if
\begin{equation}
  \label{eq:hyp-T2}
  \delta < \frac1{12} - \frac{\varpi}{2(2+\varpi)}, \qquad \lambda < \frac16 - \frac{\varpi}{2(2+\varpi)}, \qquad \rho < \frac16.
\end{equation}

We therefore conclude the following.
\begin{lemma}\label{lem:bd-T2}
  Under the notations and hypotheses of Corollary~\ref{cor:after-hb}, and assuming~\eqref{eq:hyp-T2}, we have
  $$ T_2 \ll Q^{1-\eps}\sqrt{X}. $$
  The implied constant depends on~$\lambda, \delta, \rho$ and~$\varpi$.
\end{lemma}

\subsection{Type~II sums}

In the type~II case~\eqref{eq:TII}, we wish to prove the bound
\begin{align*}
T_{\text{II}}(M, N) := \sum_{w}\Psi\left( \frac{w}{W}\right) \summ{2}{m, n} \alpha_m \beta_n \ud_R(mn, bw) \ll \sqrt{X}Q^{1-\eps},
\end{align*}
where $\alpha$ is supported at scale $M$, $\beta$ is supported at scale $N$, $MN \asymp X$, and $X^\sigma \ll N \ll X^{1/3-\delta}$. We have $|\alpha_m| \leq \tau(m)^{O(1)}$, and similarly for $\beta$. We use Linnik's dispersion method~\cite{Linnik1963a}, following closely~\cite{Fouvry1985}; see also~\cite[Section~10]{BombieriFriedlanderEtAl1986}.

We interchange the order of summation and apply the triangle inequality, writing our sum as
\begin{align*}
\abs{T_{\text{II}}(M, N)} \leq \sum_m \left|\sum_w \sum_n \right|.
\end{align*}
Applying the Cauchy-Schwarz inequality, we arrive at
\begin{equation}
  \label{eq:link-TII-D}
  T_{\text{II}}(M, N)^2 \ll M(\log M)^{O(1)} \cD,
\end{equation}
where
\begin{align*}
\cD = \sum_m f\left(\frac{m}{M} \right) \left|\summ{2}{n, w \\ mn \equiv 1 \mod{bw}}\Psi\left(\frac{w}{W}\right) \beta_{n} - \frac{1}{\varphi(bw)} \ssum{\chi\mod{bw} \\ \cond(\chi)\leq R} \summ{2}{n, w \\ (mn,bw)=1}\Psi\left(\frac{w}{W}\right)\beta_{n} \chi(mn)\right|^2.
\end{align*}
Here~$f$ is some fixed, non-negative test function majorizing~$\1_{[1, 2]}$. It suffices to show that
\begin{align*}
\mathcal{D}\ll NQ^{2-\eps}.
\end{align*}
We open the square and arrive at
\begin{equation}\label{eq:link-D-D123}
\mathcal{D} = \mathcal{D}_1 - 2\Re \mathcal{D}_2 + \mathcal{D}_3,
\end{equation}
say. We treat each sum $\mathcal{D}_i$ in turn.

\subsubsection{Evaluation of $\mathcal{D}_3$}

By definition we have
\begin{align*}
\mathcal{D}_3 &:= \sum_m f\left(\frac{m}{M} \right) \summ{4}{w_1,w_2,n_1,n_2 \\ (mn_1,bw_1)=1 \\ (mn_2,bw_2)=1} \summ{2}{\chi_1, \chi_2 \\ \chi_j \mod{bw_j} \\ \cond(\chi_j)\leq R} \Psi\left(\frac{w_1}{W} \right) \Psi\left(\frac{w_2}{W} \right) \beta_{n_1}\bar{\beta_{n_2}} \frac{\chi_1(mn_1)\bar{\chi_2(mn_2)}}{\varphi(bw_1)\varphi(bw_2)}.
\end{align*}
The computations in~\cite[p.~711--712]{Drappeau2017} can be directly quoted, putting formally
\begin{equation}
\gamma(q) = \1(b\mid q)\Psi(q/(bW)),\label{eq:link-gamma-f1}
\end{equation}
with the modification that~$\cond(\chi_1\bar{\chi_2})\leq R^2$ (instead of~$R$, as stated incorrectly in~\cite{Drappeau2017}).
Writing~$H = Q^\eps b[w_1, w_2]M^{-1}$, we get
\begin{align*}
\mathcal{D}_3 &= \mathcal{M}_3 + O\Big(Q^\eps \ssum{w_1, w_2 \asymp W \\ n_1, n_2 \asymp N} \frac1{\vphi(bw_1) \vphi(bw_2)} \ssum{\chi_1, \chi_2 \\ \cond(\chi_j) \leq R} \frac M{b[w_1, w_2]} \sum_{0<\abs{h}\leq H}R\sum_{d\mid(h, b[w_1, w_2])} d\Big) \\
&= \mathcal{M}_3 + O(Q^\eps N^2 R^5),
\end{align*}
where the main term is computed as in~\cite[p.~712]{Drappeau2017} to be
\begin{align*}
\mathcal{M}_3 := M \hat{f}(0)\summ{4}{w_1,w_2,n_1,n_2 \\ (n_j,bw_j)=1} \ssum{\chi\text{ primitive} \\ \cond(\chi)\leq R \\ \cond(\chi)\mid b(w_1, w_2)} \Psi\left(\frac{w_1}{W} \right) \Psi\left(\frac{w_2}{W} \right) \beta_{n_1}\bar{\beta_{n_2}} \chi(n_1\bar{n_2}) \frac{\varphi(bw_1w_2)}{bw_1w_2\varphi(bw_1)\varphi(bw_2)}.
\end{align*}
The error term is acceptable provided
\begin{align*}
N R^5 \ll Q^{2-\eps}.
\end{align*}
Since $N \ll X^{1/3}$ this is acceptable provided
\begin{equation}
  \label{eq:TII-hyp1}
  \rho < \frac{4-\varpi}{15(2+\varpi)}.
\end{equation}

\subsubsection{Evaluation of $\mathcal{D}_2$}

We have
\begin{align*}
\mathcal{D}_2 &:= \summ{4}{w_1,w_2,n_1,n_2 \\ (n_j,bw_j)=1} \ssum{\chi\mod{bw_2} \\ \cond(\chi)\leq R} \Psi\left(\frac{w_1}{W} \right)\Psi\left(\frac{w_2}{W} \right)\bar{\beta_{n_1}}\beta_{n_2} \frac{\chi(n_2)}{\varphi(bw_2)} \sum_{\substack{mn_1 \equiv 1 (bw_1) \\ (m,w_2)=1}} \chi(m) f\left(\frac{m}{M} \right).
\end{align*}
The computations in~\cite[p.~712--713]{Drappeau2017} can be also quoted directly with the identification~\eqref{eq:link-gamma-f1}.
We obtain
$$ \mathcal{D}_2 = \mathcal{M}_3 + O(R^{3/2}N^2 Q^{1+\eps}). $$
This is acceptable if
\begin{equation}
  \label{eq:TII-hyp2}
  \rho < \frac23\lambda + \frac{2(1-\varpi)}{9(2+\varpi)}.
\end{equation}

\subsubsection{Evaluation of $\mathcal{D}_1$}

We have
\begin{align*}
\mathcal{D}_1 &:= \summ{4}{w_1,w_2,n_1,n_2 \\ (n_j,bw_j)=1 \\ n_1 \equiv n_2 \mod{b}} \Psi\left(\frac{w_1}{W} \right) \Psi\left(\frac{w_2}{W} \right) \beta_{n_1}\bar{\beta_{n_2}} \sum_{\substack{mn_j \equiv 1 \mod{bw_j}}} f \left(\frac{m}{M} \right).
\end{align*}
We need to separate the variables $w_1,w_2,n_1,n_2$ from each other, and this requires a subdivision of the variables. We decompose these variables uniquely, following~\cite{FouvryRadziwill2018}, as follows:
\begin{align*}
\begin{cases}
d = (n_1,n_2), \\
n_1 = dd_1\nu_1, \text{ with }d_1 \mid d^\infty \text{ and } (d,\nu_1)=1, \\
n_2 = d\nu_2, \\
q_0 = (w_1,w_2), \\
w_i = q_0 q_i, i \in \{1,2\}.
\end{cases}
\end{align*}
The summation conditions imply
\begin{align*}
(dd_1\nu_1,q_0q_1)=(d\nu_2,q_0q_2)=1.
\end{align*}
We therefore have
\begin{align*}
\mathcal{D}_1 ={}& \sum_{(d,b)=1} \sum_{d_1 \mid d^\infty} \sum_{(q_0,d)=1} \mathcal{D}_1(d, d_1, q_0), \\
\mathcal{D}_1(\cdots) ={}& \summ{4}{q_1,q_2,\nu_1,\nu_2  \\ (d\nu_1, \nu_2) = (q_1, q_2) = 1 \\ (q_1q_2, d) = (\nu_1, d) = 1 \\ (\nu_1, q_1) = (\nu_2, q_2) = (\nu_1\nu_2, bq_0) = 1 \\ d_1\nu_1 \equiv \nu_2 \mod{bq_0}} \Psi\left(\frac{q_0q_1}{W} \right) \Psi\left(\frac{q_0q_2}{W} \right)\beta_{dd_1\nu_1}\bar{\beta_{d\nu_2}} \sum_{\substack{mdd_1\nu_1 \equiv 1 \mod{bq_0q_1} \\ md\nu_2 \equiv 1 \mod{bq_0q_2}}} f \left(\frac{m}{M} \right).
\end{align*}

Using smooth partitions of unity we break the variables into dyadic ranges: $d \asymp D, d_1 \asymp D_1, q_0 \asymp Q_0$. The contribution from $d \asymp D$ and  $d_1 \asymp D_1$ is
\begin{align*}
&\ll Q^\eps M \sum_{d \asymp D} \sum_{\substack{d_1 \mid d^\infty \\ d_1 \asymp D_1}} \sum_{\nu_1 \asymp N/dd_1}\sum_{\nu_2 \asymp N/d} |\beta_{dd_1\nu_1}| |\beta_{d\nu_2}| \ll Q^\eps M N^2\sum_{d \asymp D} \frac{1}{d^2} \sum_{\substack{d_1 \mid d^\infty }} \frac{\tau(d_1)^{O(1)}}{d_1} \left(\frac{d_1}{D_1} \right)^{1-\eps^2} \\
&\ll Q^\eps M N^2 D_1^{-1+\eps^2}D^{-1},
\end{align*}
where the sum over~$q_0, q_1$ was bounded by~$O(\tau_3(\abs{mdd_1\nu_1-1})) = O(Q^\eps)$, likewise for the sum over~$q_2$ (note that~$md\nu_2\neq 1$ and~$mdd_1\nu_1\neq 1$). This bound is acceptable provided
\begin{equation}\label{eq:TII-restriction-DD1}
DD_1 \gg \frac{X}{Q^{2-\eps}},
\end{equation}
so we may henceforth assume $DD_1 \ll X Q^{-2+\eps}$.

The contribution from $q_0 \asymp Q_0$ is
\begin{align*}
&\ll Q^\eps \sum_{q_0 \asymp Q_0}\sum_{q_1 \asymp Q/q_0} \summ{2}{n_1 \equiv n_2 \mod{q_0} \\ n_j \asymp N}\ssum{m\asymp M \\ m \equiv \overline{n_1}\mod{q_0q_1}} 1 \\
&\ll Q^\eps M \sum_{q_0 \asymp Q_0} \sum_{q_1 \asymp Q/q_0} \frac{1}{q_0q_1}\summ{2}{n_1 \equiv n_2 \mod{q_0} \\ n_j \asymp N} 1 \\
&\ll Q^\eps \left(MN^2 Q_0^{-1} + MN \right),
\end{align*}
where in the first line the sum over~$q_2$ was again bounded by~$\tau(\abs{md\nu_2-1})$. This is acceptable provided
\begin{equation}
  \label{eq:TII-hyp3}
  N \gg \frac{X}{Q^{2-\eps}}, \quad \text{ and } \quad Q_0 \gg \frac{X}{Q^{2-\eps}},
\end{equation}
so we may henceforth assume $Q_0 \ll X Q^{-2+\eps}$.

We use Poisson summation, following~\cite[pp.~714--716]{Drappeau2017}. Let
$$ \qt = bq_0q_1q_2, \qquad \mu \equiv
\begin{cases}
\bar{dd_1\nu_1} \pmod{bq_0q_1}, \\
\bar{d\nu_2} \pmod{bq_0q_2}.
\end{cases} $$
Note that~$\qt \geq \frac12 W \gg Q^{1-\eps}$. With~$H = \qt^{1+\eps}M^{-1} \ll Q^{2+\eps}/(q_0M)$, we get for any fixed~$A>0$
\begin{equation}
\sum_{m \equiv \mu\mod{\qt}} f\Big(\frac mM\Big) = \frac{M}{\qt} \sum_{\abs{h}\leq H} {\hat f}\Big(\frac{hM}{\qt}\Big)\e\Big(\frac{\mu h}{\qt}\Big) + O(Q^{-A}).\label{eq:TII-afterpoisson}
\end{equation}

The zero frequency in~\eqref{eq:TII-afterpoisson} contributes the main term, which, after summing over~$d, d_1, q_0$ (and reintegrating the values~$DD_1$, $Q_0$ larger than~$XQ^{-2+\eps}$ which were discarded earlier), is given by
$$ \mathcal{M}_1 := \frac Mb {\hat f}(0)\summ{4}{w_1,w_2,n_1,n_2 \\ (n_j,bw_j)=1 \\ n_1 \equiv n_2 \mod{b(w_1, w_2)}} \Psi\left(\frac{w_1}{W} \right) \Psi\left(\frac{w_2}{W} \right) \beta_{n_1}\bar{\beta_{n_2}} \frac1{[w_1, w_2]}. $$
The error term in~\eqref{eq:TII-afterpoisson} induces in~$\mathcal{D}_1(d, d_1, q_0)$ a contribution
$$ \ll Q^{-10} N^2, $$
and therefore in~$\mathcal{D}_1$ a contribution~$O(1)$, which is acceptable.

We solve the congruence conditions on~$\mu$ by writing
$$ d_1\nu_1 - \nu_2 = bq_0 t, \qquad \mu d d_1 \nu_1 = 1 + bq_0q_1 \ell, \qquad \mu d \nu_2 = 1 + bq_0q_2m, $$
with~$t, \ell, m \in \Z$. We deduce
$$ \mu d t = q_1\ell - q_2 m, \qquad t = q_1\nu_2\ell - q_2d_1\nu_1m. $$
Then we have the equalities, modulo~$\Z$,
\begin{align*}
  \frac{\mu}{\qt} = \frac\mu{bq_0q_1q_2} ={}& \frac1{dd_1\nu_1bq_0q_1q_2} + \frac{\ell}{dd_1\nu_1q_2} \\
  \equiv {}&  \frac1{dd_1\nu_1bq_0q_1q_2} + \frac{\ell\bar{dd_1}}{\nu_1q_2} + \frac{\ell\bar{\nu_1q_2}}{dd_1} \\
  \equiv {}&  \frac1{dd_1\nu_1bq_0q_1q_2} + \frac{t\bar{q_1\nu_2dd_1}}{\nu_1q_2} - \frac{\bar{bq_0q_1\nu_1q_2}}{dd_1} \\
  \equiv {}&  \frac1{dd_1\nu_1bq_0q_1q_2} + \frac{d_1\nu_1 - \nu_2}{bq_0} \frac{\bar{q_1\nu_2dd_1}}{\nu_1q_2} - \frac{\bar{bq_0q_1\nu_1q_2}}{dd_1}.
\end{align*}
By estimating trivially the first term, we have
\begin{equation}
\e\Big(\frac{h\mu}{\qt}\Big) = \e\Big(h\frac{d_1\nu_1 - \nu_2}{bq_0} \frac{\bar{q_1\nu_2dd_1}}{\nu_1q_2} - \frac{h\bar{bq_0q_1\nu_1q_2}}{dd_1}\Big) + O\Big(\frac{Hq_0}{NW^2}\Big).\label{eq:TII-decomp-exp}
\end{equation}
The error term here is~$\ll Q^\eps X^{-1}$, which contributes to~$\mathcal{D}_1(d, d_1, q_0)$ a quantity
$$ \frac{Q^{2+\eps}N}{X q_0^2 d d_1} \Big(1 + \frac Nd\Big), $$
and upon summing over~$(d, d_1, q_0)$, this contributes to~$\mathcal{D}_1$ a quantity~$O(Q^{2+\eps}N^2 X^{-1})$. This error is acceptable if
\begin{equation}
N \ll Q^{2-\eps}.\label{eq:TII-hyp4}
\end{equation}

Then we insert the first term of~\eqref{eq:TII-decomp-exp} in~\eqref{eq:TII-afterpoisson}, and insert the Fourier integral. The non-zero frequencies contribute a term
\begin{align*}
  \mathcal{R}_1(d, d_1, q_0) :={}& \frac{Mq_0}{bW^2} \int \summ{4}{q_1,q_2,\nu_1,\nu_2  \\ (d\nu_1, \nu_2) = (q_1, q_2) = 1 \\ (q_1q_2, d) = (\nu_1, d) = 1 \\ (\nu_1, q_1) = (\nu_2, q_2) = (\nu_1\nu_2, bq_0) = 1 \\ d_1\nu_1 \equiv \nu_2 \mod{bq_0}} \sum_{0<\abs{h}\leq H}
\Psi\left(\frac{q_0q_1}{W} \right) \Psi\left(\frac{q_0q_2}{W} \right)\beta_{dd_1\nu_1}\bar{\beta_{d\nu_2}}\times \\
{}& \qquad\qquad \times  f\Big(t\frac{q_0^2q_1q_2}{W^2}\Big) \e\Big(h\frac{d_1\nu_1 - \nu_2}{bq_0} \frac{\bar{q_1\nu_2dd_1}}{\nu_1q_2} - \frac{h\bar{bq_0q_1\nu_1q_2}}{dd_1}\Big) \e\Big(\frac{-htMq_0}{bW^2}\Big)\df t.
\end{align*}
So far, we have obtained under the conditions~\eqref{eq:TII-hyp3} and \eqref{eq:TII-hyp4} the bound
$$ \mathcal{D}_1 = \mathcal{M}_1 + \mathcal{R}_1  + O(NQ^{2-\eps}), $$
$$ \mathcal{R}_1 := \ssum{Q_0, DD1\ll XQ^{-2+\eps} \\ Q, D, D_1 \text{ dyadic}} \ssum{d\asymp D \\d_1\asymp D_1 \\q_0 \asymp Q_0} \mathcal{R}(d, d_1, q_0). $$

We now restrict the summation over~$q_1, q_2$ in residue classes modulo~$dd_1$, to account for the oscillatory factors.
Let~$\lambda_1, \lambda_2\in(\Z/dd_1\Z)^\times$, and
$$ b_{\bn,\br,\bs} = \mathop{\sum_{\nu_1} \sum_{\nu_2}}_{\substack{\nu_1 = \bs \\ \nu_2 d d_1 = \br \\ (d\nu_1, \nu_2) = (\nu_1\nu_2, bq_0) = 1 \\ (\nu_1, d) = 1 \\ d_1\nu_1 \equiv \nu_2 \pmod{bq_0}}} \sum_{\substack{0<\abs{h}\leq H \\ h(d_1\nu_1 - \nu_2) = bq_0 \bn}} \beta_{dd_1\nu_1} \bar{\beta_{d\nu_2}}  \e\Big(- \frac{h\bar{bq_0\lambda_1\nu_1\lambda_2}}{dd_1} - \frac{htMq_0}{bW^2}\Big), $$
$$ g(\bc,\bd,\bn,\br,\bs) = \Psi\Big(\frac{q_0\bc}{W}\Big) \Psi\Big(\frac{q_0\bd}{W}\Big) f\Big(\frac{t q_0^2\bc\bd}{W^2}\Big). $$
Then
$$ \mathcal{R}_1(d, d_1, q_0) = \frac{Mq_0}{bW^2} \int_{t\asymp_f 1} \sum_{\lambda_1, \lambda_2 \pmod{dd_1}^\ast} \tilde{\mathcal{R}}_1(t, (\lambda_j))\df t, $$
$$ \tilde{\mathcal{R}}_1(t, (\lambda_j)) = \sum_{\substack{\bn, \br, \bs, \bc, \bd \\ \bc\equiv\lambda_1,\ \bd\equiv \lambda_2 \pmod{dd_1} \\ (\bs\bc, \br\bd bdd_1)=1}} b_{\bn,\br,\bs} g(\bc,\bd,\bn,\br,\bs) \e\Big(\frac{\bn\bar{\br\bd}}{\bs\bc}\Big). $$
We apply Proposition~\ref{prop:newbound-DI}, with sizes given by
$$ \bC = \bD = \frac{W}{q_0}, \quad \bS = \frac{N}{dd_1}, \quad \bR = Nd_1, \quad \bN = \frac{HN}{dbq_0}. $$
Let~$X = Q^2Y$, so~$Y = Q^{\varpi}$. Note that
$$ \bR\bS \asymp N^2 D^{-1}, \qquad \bN \ll Q^\eps N^2Y^{-1}D^{-1}Q_0^{-2} \ll Q^\eps \bR\bS, \qquad \bC \ll Q^\eps \bR \bD. $$
We get
$$ \tilde{\mathcal{R}}_1(t, \lambda_j) \ll Q^\eps (DD_1)^{3/2} K \|b_{\bN,\bR,\bS}\|_2, $$
where
\begin{align*}
Q^{-\eps}K^2 \ll {}& Q^2N^4D^{-1}D_1Q_0^{-2} + Q^{2+4\theta}N^{4-6\theta} D^{-2+2\theta} D_1^{-2\theta} Q_0^{-2-4\theta}\red{(1+D/N)^{1-4\theta}} + Q^2N^3 Y^{-1} D^{-1} D_1 Q_0^{-4}.
\end{align*}
To bound the term~$\|b_{\bN,\bR,\bS}\|_2$, we assume
\begin{equation}
  X Q^{-2+\eps} = o(N),\label{eq:TII-hyp5}
\end{equation}
so that~$D = o(N)$ by virtue of the line below~\eqref{eq:TII-restriction-DD1}, and the case~$d_1\nu_1 = \nu_2$ never occurs in~$b_{\bn,\br,\bs}$. Then
\begin{align*}
\|b_{\bN,\bR,\bS}\|_2^2 &\leq \sum_{\substack{\nu_1, \nu_2, h \\ d_1\nu_1 \equiv \nu_2\mod{q_0} \\ 0<\abs{h}<H}} \abs{\beta_{dd_1\nu_1}\beta_{d\nu_2}}^2 \ll \frac{Q^{2+\varepsilon}}{Q_0 M}\frac{N}{D D_1}\left(\frac{N}{D Q_0} + 1 \right) \\
&\ll Q ^\eps (N^3 Y^{-1}D^{-2}D_1^{-1}Q_0^{-2} + N^2 Y^{-1}D^{-1}D_1^{-1}Q_0^{-1}).
\end{align*}
We deduce
$$ \tilde{\mathcal{R}}_1(t, (\lambda_j)) \ll Q^\eps \sum_{k=1}^6 Q^{\eta_{k,1}} N^{\eta_{k,2}} Y^{\eta_{k,3}} D^{\eta_{k,4}} D_1^{\eta_{k,5}} Q_0^{\eta_{k,6}}, $$
where for each~$k$,~$\eta_k = (\eta_{k, \ell})_{1\leq \ell \leq 6}$ is given by
\[
\{\eta_k\} = \left\{
  \pmat{1 \\ 3 \\ -1/2 \\ 1/2 \\ 3/2 \\ -3/2},
  \pmat{1 \\ 7/2 \\ -1/2 \\ 0 \\ 3/2 \\ -2},
  \pmat{2\theta + 1 \\ 3-3\theta \\ -1/2 \\ \theta \\ 1-\theta \\ -2\theta - 3/2},
  \pmat{2\theta + 1 \\ 7/2-3\theta  \\ -1/2 \\ \theta - 1/2 \\ 1-\theta \\ -2\theta - 2},
  \pmat{1 \\ 5/2 \\ -1 \\ 1/2 \\ 3/2 \\ -5/2},
  \pmat{1 \\ 3 \\ -1 \\ 0 \\ 3/2 \\ -3}
\right\}.
\]
Summing over~$\lambda_j$, integrating over~$t$, and multiplying by~$\frac{Mq_0}{bW^2} \ll Q^\eps N^{-1}Y Q_0$, we get
$$ \mathcal{R}_1(d, d_1, q_0) \ll Q^\eps \sum_{k=1}^6 Q^{\eta_{k,1}} N^{\eta_{k,2}-1} Y^{\eta_{k,3}+1} D^{\eta_{k,4}+2} D_1^{\eta_{k,5}+2} Q_0^{\eta_{k,6}+1}. $$
We sum over~$d, d_1$ and~$q_0$ in dyadic intervals of lengths~$D, D_1$ and~$Q_0$, obtaining
\begin{align*}
\ssum{d \asymp D \\ d_1 \asymp D_1,\ d_1 \mid d^\infty \\ q_0 \asymp Q_0 \\ (d, b) = (q_0, d) = 1} \mathcal{R}_1(d, d_1, q_0) \ll  Q^\eps \sum_{k=1}^6 Q^{\eta_{k,1}} N^{\eta_{k,2}-1} Y^{\eta_{k,3}+1} D^{\eta_{k,4}+3} D_1^{\eta_{k,5}+2} Q_0^{\eta_{k,6}+2}.
\end{align*}
Finally we sum this dyadically over~$Q_0, D, D_1$ subject to~$Q_0 + DD_1 \ll YQ^\eps$. We get
$$ \mathcal{R}_1 \ll Q^\eps \sum_{k=1}^6 Q^{\eta_{k,1}} N^{\eta_{k,2}-1} Y^{\eta_{k,3}+1 +\max(0, \eta_{k,6}+2) + \max(0, \eta_{k,4}+3, \eta_{k,5}+2)}. $$
Here, the terms for~$k=5, 6$ are majorized by the term~$k=1$, therefore,
$$ \mathcal{R}_1 \ll Q^\eps \sum_{k=1}^4 Q^{\theta_{k,1}} N^{\theta_{k,2}} Y^{\theta_{k,3}}, $$
where
\[
\{\theta_k\} = \left\{
\pmat{1 \\ 2 \\ 9/2},
\pmat{1 \\ 5/2 \\ 4},
\pmat{1+2\theta \\ 2-3\theta \\ 4-\theta},
\pmat{1+2\theta \\ 5/2-3\theta \\ 7/2-\theta}
\right\}.
\]
We conclude that
$$ \mathcal{D}_1 = \mathcal{M}_1 + O(Q^{2-\eps}N) $$
on the condition~$N \ll Q^{-\eps} \min(Q Y^{-9/2}, Q^{2/3} Y^{-8/3}, Q^{\frac{1-2\theta}{1-3\theta}} Y^{-\frac{4-\theta}{1-3\theta}}, Q^{2/3}Y^{-\frac{7-2\theta}{3(1-2\theta)}})$. Upon using~$\theta\leq 7/64$, these conditions are implied by
\begin{equation}
N \ll X^{-\eps}\min(X^{\frac{2-9\varpi}{2(2+\varpi)}}, X^{\frac{50-249\varpi}{43(2+\varpi)}}, X^{\frac{50-217\varpi}{75(2+\varpi)}}),\label{eq:TII-hyp6}
\end{equation}
and the above hypotheses~\eqref{eq:TII-hyp3}, \eqref{eq:TII-hyp4}, \eqref{eq:TII-hyp5}.

\subsubsection{Main terms}

The main terms~$\mathcal{M}_1$ and~$\mathcal{M}_3$, which are real numbers by the symmetry~$n_1\leftrightarrow n_2$, combine to form
\begin{align*}
  \mathcal{M}_1 - \mathcal{M}_3 = M\hat{f}(0)\summ{2}{w_1, w_2} {}& \Psi\Big(\frac{w_1}W\Big) \Psi\Big(\frac{w_2}W\Big) \frac{1}{b[w_1,w_2]\vphi(b(w_1, w_2))} \\
& \qquad \times\ssum{\chi\text{ prim} \\ \cond(\chi)> R \\ \cond(\chi)\mid b(w_1, w_2)} \summ{2}{n_1, n_2 \\ (n_j, bw_j)=1} \beta_{n_1}\bar{\beta_{n_2}} \chi(n_1)\bar{\chi(n_2)}.
\end{align*}
We may quote the computations in~\cite[p.~717]{Drappeau2017}, again with the identification~\eqref{eq:link-gamma-f1}, to obtain
$$ \abs{\mathcal{M}_3 - \mathcal{M}_1} \ll Q^\eps M(N + N^2 R^{-2}) \ll Q^\eps(X + NX R^{-2}). $$
This is acceptable provided
\begin{equation}
N \gg Q^{\varpi + \eps}, \qquad R \gg Q^{\frac{\varpi}2 + \eps}.\label{eq:TII-hyp7}
\end{equation}

\subsubsection{Conclusion}

The hypotheses~\eqref{eq:TII-hyp1}, \eqref{eq:TII-hyp2}, \eqref{eq:TII-hyp3}, \eqref{eq:TII-hyp4}, \eqref{eq:TII-hyp5}, \eqref{eq:TII-hyp6} and \eqref{eq:TII-hyp7} are all satisfied if
\begin{equation}
  \label{eq:hyp-TII}
  \varpi < 1/8, \qquad \varpi < \sigma < \frac13-\delta < \frac13 - \frac{242\varpi}{75(2+\varpi)}, \qquad \frac{\varpi}{2(2+\varpi)} < \rho < \frac19 - \frac{\varpi}{3(2+\varpi)}.
\end{equation}
We therefore conclude the following.
\begin{lemma}\label{lem:bd-TII}
  Under the notations and hypotheses of Corollary~\ref{cor:after-hb}, assuming~\eqref{eq:hyp-TII}, we have
  $$ T_{\text{II}} \ll \sqrt{X} Q^{1-\eps}. $$
\end{lemma}

\subsection{Proof of Proposition~\ref{prop:primes-ap}}

We combine Lemmas~\ref{lem:bd-T1}, \ref{lem:bd-T2}, \ref{lem:bd-TII} and~\ref{lemma:combinatorial lemma}. Setting~$\sigma = \varpi + \eps$ and recalling that~$\varpi < 1/8$, we obtain the conditions
\begin{align*}
  \frac{\varpi}{3(2+\varpi)} < \lambda &{} < \frac16 - \frac\varpi2, \\ 
  \frac{242\varpi}{75(2+\varpi)} < \delta & {}< \frac1{12} - \frac{\varpi}{2(2+\varpi)}, \\
  \frac{\varpi}{2(2+\varpi)} < \rho &{}< \frac19 - \frac{\varpi}{3(2+\varpi)}.
\end{align*}
The third is implied by our hypothesis on~$R$. The first two can be satisfied whenever~$-o(1) \leq \varpi < \frac{50}{1093}- o(1)$. This proves Proposition~\ref{prop:primes-ap}.

\bibliographystyle{amsalpha2}
\bibliography{bib}

\newcommand{\etalchar}[1]{$^{#1}$}
\providecommand{\bysame}{\leavevmode\hbox to3em{\hrulefill}\thinspace}
\providecommand{\MR}{\relax\ifhmode\unskip\space\fi MR }
\providecommand{\MRhref}[2]{%
  \href{http://www.ams.org/mathscinet-getitem?mr=#1}{#2}
}
\providecommand{\href}[2]{#2}
\begin{thebibliography}{MMR{\etalchar{+}}16}

\bibitem[BS15]{Bhargava2015}
M.~Bhargava and A.~Shankar, \emph{Ternary cubic forms having bounded
  invariants, and the existence of a positive proportion of elliptic curves
  having rank 0}, Ann. of Math. (2) (2015), 587--621.

\bibitem[BFI]{BombieriEtAl2019}
E.~Bombieri, J.~B. Friedlander, and H.~Iwaniec, \emph{Some corrections to an
  old paper}, ArXiv e-print 1903.01371.

\bibitem[BFI86]{BombieriFriedlanderEtAl1986}
\bysame, \emph{Primes in arithmetic progressions to large moduli}, Acta Math.
  \textbf{156} (1986), no.~3-4, 203--251.

\bibitem[Bru92]{Brum92}
A.~Brumer, \emph{The average rank of elliptic curves. {I}}, Invent. Math.
  \textbf{109} (1992), no.~3, 445--472.

\bibitem[BM11]{BuiMilinovich2011}
H.~M. Bui and Micah~B. Milinovich, \emph{Central values of derivatives of
  {D}irichlet {$L$}-functions}, Int. J. Number Theory \textbf{7} (2011), no.~2,
  371--388.

\bibitem[CLLR14]{ChandeeRadziwill}
V.~Chandee, Y.~Lee, S.-C. Liu, and M.~Radziwi\l{}\l{}, \emph{Simple zeros of
  primitive {D}irichlet {$L$}-functions and the asymptotic large sieve}, Q. J.
  Math. \textbf{65} (2014), no.~1, 63--87.

\bibitem[CI02]{ConreyIwaniec}
B.~Conrey and H.~Iwaniec, \emph{Spacing of zeros of {H}ecke {$L$}-functions and
  the class number problem}, Acta Arith. \textbf{103} (2002), no.~3, 259--312.

\bibitem[DI82]{DeshouillersIwaniec1982a}
J.-M. Deshouillers and H.~Iwaniec, \emph{Kloosterman sums and {Fourier}
  coefficients of cusp forms}, Invent. Math. \textbf{70} (1982), no.~2,
  219--288.

\bibitem[DFS]{DevinFiorilliSodergren}
L.~Devin, D.~Fiorilli, and A.~S\"odergren, \emph{Low-lying zeros in families of
  holomorphic cusp forms: the weight aspect}, ArXiv e-print 1911.08310.

\bibitem[Dra15]{Drappeau2015}
S.~Drappeau, \emph{Th\'{e}or\`emes de type {F}ouvry-{I}waniec pour les entiers
  friables}, Compos. Math. \textbf{151} (2015), no.~5, 828--862.

\bibitem[Dra17]{Drappeau2017}
\bysame, \emph{Sums of {K}loosterman sums in arithmetic progressions, and the
  error term in the dispersion method}, Proc. London Math. Soc. (3)
  \textbf{114} (2017), no.~4, 684--732.

\bibitem[FM15]{FiorilliMiller}
D.~Fiorilli and S.~J. Miller, \emph{Surpassing the ratios conjecture in the
  1-level density of {D}irichlet {$L$}-functions}, Algebra Number Theory
  \textbf{9} (2015), no.~1, 13--52.

\bibitem[Fou85]{Fouvry1985}
\'E. Fouvry, \emph{Sur le probl\`eme des diviseurs de {Titchmarsh}}, J. Reine
  Angew. Math. \textbf{357} (1985), 51--76.

\bibitem[FI83]{FouvryIwaniec1983}
\'E. Fouvry and H.~Iwaniec, \emph{Primes in arithmetic progressions}, Acta
  Arith. \textbf{42} (1983), no.~2, 197--218.

\bibitem[FI03]{Fouvry2003}
\bysame, \emph{Low-lying zeros of dihedral {$L$}-functions}, Duke Math. J.
  \textbf{116} (2003), no.~2, 189--217.

\bibitem[FR18]{FouvryRadziwill2018}
\'E. Fouvry and M.~Radziwi\l{}\l{}, \emph{Another application of {L}innik's
  dispersion method}, Chebyshevski\u{\i} Sb. \textbf{19} (2018), no.~3,
  148--163.

\bibitem[GS18]{GS18}
A.~Granville and K.~Soundararajan, \emph{Large character sums: {B}urgess's
  theorem and zeros of {$L$}-functions}, J. Eur. Math. Soc. (JEMS) \textbf{20}
  (2018), no.~1, 1--14.

\bibitem[HB82]{Heath-Brown1982}
D.~R. Heath-Brown, \emph{Prime numbers in short intervals and a generalized
  {Vaughan} identity}, Canad. J. Math. \textbf{34} (1982), no.~6, 1365--1377.

\bibitem[HB04]{HeathBrown}
\bysame, \emph{The average analytic rank of elliptic curves}, Duke Math. J.
  \textbf{122} (2004), no.~3, 591--623.

\bibitem[HR03]{HughesRudnick}
C.~P. Hughes and Z.~Rudnick, \emph{Linear statistics of low-lying zeros of
  {$L$}-functions}, Q. J. Math. \textbf{54} (2003), no.~3, 309--333.

\bibitem[IK04]{IwaniecKowalski2004}
H.~Iwaniec and E.~Kowalski, \emph{Analytic number theory}, American
  Mathematical Society Colloquium Publications, vol.~53, American Mathematical
  Society, Providence, RI, 2004.

\bibitem[ILS00]{IwaniecLuoSarnak}
H.~Iwaniec, W.~Luo, and P.~Sarnak, \emph{Low lying zeros of families of
  {$L$}-functions}, Inst. Hautes \'{E}tudes Sci. Publ. Math. (2000), no.~91,
  55--131.

\bibitem[KS99]{KatzSarnak}
N.~M. Katz and P.~Sarnak, \emph{Random matrices, {F}robenius eigenvalues, and
  monodromy}, American Mathematical Society Colloquium Publications, vol.~45,
  American Mathematical Society, Providence, RI, 1999.

\bibitem[KS03]{KimSarnak2003}
H.~Kim and P.~Sarnak, \emph{Refined estimates towards the {Ramanujan} and
  {Selberg} conjectures}, J. Amer. Math. Soc \textbf{16} (2003), no.~1,
  175--181.

\bibitem[Lin63]{Linnik1963a}
Ju.~V. Linnik, \emph{The dispersion method in binary additive problems},
  Translated by S. Schuur, American Mathematical Society, Providence, R.I.,
  1963.

\bibitem[MMR{\etalchar{+}}16]{survey}
B.~Mackall, S.~J. Miller, C.~Rapti, C.~Turnage-Butterbaugh, and K.~Winsor,
  \emph{Some results in the theory of low-lying zeros of families of
  {$L$}-functions}, Families of automorphic forms and the trace formula, Simons
  Symp., Springer, [Cham], 2016, pp.~435--476.

\bibitem[Mon73]{Montgomery}
H.~L. Montgomery, \emph{The pair correlation of zeros of the zeta function},
  Amer. Math. Soc., Providence, R.I., 1973.

\bibitem[Pra19]{Pratt2018}
K.~Pratt, \emph{Average non-vanishing of {D}irichlet {$L$}-functions at the
  central point}, Algebra Number Theory \textbf{13} (2019), no.~1, 227--249.

\bibitem[RS94]{RudnickSarnak}
Z.~Rudnick and P.~Sarnak, \emph{The {$n$}-level correlations of zeros of the
  zeta function}, Comptes Rendus de l'Acad\'{e}mie des Sciences. S\'{e}rie I.
  Math\'{e}matique \textbf{319} (1994), no.~10, 1027--1032.

\bibitem[RS96]{RudnickSarnak1996}
\bysame, \emph{Zeros of principal {$L$}-functions and random matrix theory},
  Duke Math. J. \textbf{81} (1996), no.~2, 269--322, A celebration of John F.
  Nash, Jr.

\bibitem[SST16]{SarnakShinTemplier2016}
P.~Sarnak, S.~W. Shin, and N.~Templier, \emph{Families of {$L$}-functions and
  their symmetry}, Families of automorphic forms and the trace formula, Simons
  Symp., Springer, [Cham], 2016, pp.~531--578.

\bibitem[Sic98]{Sica1998}
F.~Sica, \emph{The order of vanishing of {L}-functions at the center of the
  critical strip}, Ph.D. thesis, McGill University (Canada), 1998.

\bibitem[Wat21]{Wat21}
M.~Watkins, \emph{Comments on {D}euring's zero-spacing phenomenon}, J. Number
  Theory \textbf{218} (2021), 1--43.

\bibitem[You06]{Young}
M.~P. Young, \emph{Low-lying zeros of families of elliptic curves}, J. Amer.
  Math. Soc. \textbf{19} (2006), no.~1, 205--250.

\bibitem[Zha14]{Zhang2014}
Y.~Zhang, \emph{Bounded gaps between primes}, Ann. of Math. (2) \textbf{179}
  (2014), no.~3, 1121--1174.

\end{thebibliography}

\end{document}